\documentclass[draftcls,onecolumn,10pt]{IEEEtran}
\usepackage{array}
\usepackage{multirow}
\newtheorem{theorem}{Theorem}

\newtheorem{lemma}{Lemma}
\usepackage{epsfig}
\usepackage{amsmath}
\usepackage{amstext}
\usepackage{amsfonts}
\usepackage{amssymb}
\usepackage{eucal}
\usepackage{graphicx}
\usepackage{verbatim}
\usepackage{stmaryrd}
\usepackage{pstricks}
\usepackage{pst-node}
\usepackage{psfrag}
\usepackage{bm}

\usepackage{tikz}
\usepackage{pgfplots}
\usepackage{algorithmic}
\newcommand{\RR}{{\mathbb{R}}}

\newcommand{\CC}{{\mathbb{C}}}
\newcommand{\herm}{{\sf H}}
\newcommand{\trans}{{\sf T}}

\newcommand{\M}{{\bf M}}
\newcommand{\V}{{\bf V}}
\newcommand{\D}{{\bf D}}
\newcommand{\R}{{\bf R}}
\newcommand{\X}{{\bf X}}
\newcommand{\Z}{{\bf Z}}
\newcommand{\G}{{\bf G}}
\newcommand{\Q}{{\bf Q}}
\newcommand{\Y}{{\bf Y}}
\newcommand{\W}{{\bf W}}
\newcommand{\A}{{\bf A}}
\newcommand{\B}{{\bf B}}
\newcommand{\T}{{\bf T}}

\newcommand{\U}{{\bf U}}

\renewcommand{\H}{{\bf H}}
\newcommand{\I}{{\bf I}}

\newcommand{\y}{{\bf y}}

\DeclareMathOperator{\tr}{tr}

\newcommand{\asto}{\overset{\rm a.s.}{\longrightarrow}}

\newcommand{\EE}{\mathbb{E}}

\newcommand{\trad}{\hat{\mathcal I}_{\mathrm{SE}}}
\newcommand{\gest}{\hat{\mathcal I}_{\mathrm{G}}}
\newcommand{\yest}{\hat{y}_{N,t}}
\newcommand{\erg}{{\mathcal I}}
\newcommand{\II}{{\mathcal I}}
\newcommand{\GG}{\boldsymbol{\Gamma}}

\newtheorem{remark}{Remark}
\newtheorem{assumption}{{\bf Assumption}}

\newtheorem{proposition}{Proposition}

\IEEEoverridecommandlockouts

\author{Abla Kammoun$^{1}$, Romain Couillet$^{2}$, Jamal Najim$^1$ and M{\'e}rouane~Debbah$^2$\\ {\it $^1$Telecom ParisTech, $^2$Alcatel-Lucent Chair, Sup\'elec.}}
	
	\title{Performance of mutual information inference methods under unknown interference}
	
\begin{document}

	\maketitle
	\begin{abstract}
          In this paper, the problem of fast point-to-point MIMO channel mutual information estimation is addressed, in the situation where the receiver undergoes unknown colored interference, whereas the channel with the transmitter is perfectly known. The considered scenario  assumes that the estimation is based on a few channel use observations during a short sensing period. Using large dimensional random matrix theory, an estimator referred to as {\em G-estimator} is derived. This estimator is proved to be consistent as the number of antennas and observations grow large and its asymptotic performance is analyzed. In particular, the G-estimator satisfies a central limit theorem with asymptotic Gaussian fluctuations. Simulations are provided which strongly support the theoretical results, even for small system dimensions.
	\end{abstract}

	\section{Introduction}
	The use of multiple-input-multiple-output (MIMO) technologies has the potential to achieve high data rates, since several independent channels between the transmitter and the receiver can be exploited. However, the proper evaluation of the achievable rate in the MIMO setting is fundamentally contingent to the knowledge of the transmit-receive channel as well as of the interference pattern. In recent communication schemes such as cognitive radios \cite{MIT99}, it is fundamental for a receiver to be able to infer these achievable rates in a short sensing period, hence extremely fast. This article is dedicated to the study of novel algorithms that partially fulfill this task without resorting to the (usually time consuming) evaluation of the covariance matrix of the interference.
	
	Conventional methods for the estimation of the mutual information in single antenna systems rely on the use of classical estimation techniques which assume a large number of observations. In general, consider $\theta$ a parameter we wish to estimate, and $M$ the number of independent and identically distributed (i.i.d.) observation vectors $\y_1,\cdots,\y_M\in \CC^{N}$. Assume $\theta$ is a function of the covariance matrix ${\boldsymbol\Sigma}=\EE\left[\y_1\y_1^{\herm}\right]$ of the received random process, i.e. $\theta=f({\boldsymbol\Sigma})$, for some function $f$. From the strong law of large numbers, a consistent estimate of the covariance of the random process is simply given by the empirical covariance of $\Y=\left[\y_1,\cdots,\y_M\right]$, i.e. $\widehat{\boldsymbol\Sigma}\triangleq \frac{1}{M}\Y\Y^{\herm}=\frac{1}{M}\sum_{i=1}^M \y_i\y_i^{\herm}$. The one-step estimator $\widehat{\theta}$ of $\theta$ would then consist in using the empirical covariance matrix $\widehat{\boldsymbol\Sigma}$ as a good approximation of ${\boldsymbol\Sigma}$, thus yielding $\widehat{\theta}=f(\widehat{\boldsymbol\Sigma})$ \cite{VAN00}. Such methods provide good performance as long as the number of observations $M$ is very large compared to the vector size $N$, a situation not always encountered in wireless communications, especially in fast changing channel environments.

	To address the scenario where the number of observations $M$ is of the same order as the dimension $N$ of each observation, new consistent estimation methods, sometimes called G-estimation methods (named after Girko's pioneering works \cite{GIR00,GIR98} on {\em General Statistical Analysis}) have been developed, mainly based on large dimensional random matrix theory. In the context of wireless communications, works devoted to the estimation of eigenvalues and eigenspace projections \cite{MES08,MES08b} have given rise to improved subspace estimation techniques \cite{LOU10,MES08c}. Recently, the use of these methods to better estimate system performance indexes in wireless communications has triggered the interest of many researchers. In particular, the estimation of the mutual information of MIMO systems under imperfect channel knowledge has been addressed in \cite{RYA08} and \cite{LOU09b}, where methods based respectively on free probability theory and the Stieltjes transform were proposed.
	
	In this article, we consider a different situation where the receiver perfectly knows the channel with the transmitter but does not a priori know the experienced interference. Such a situation can be encountered in multi-cell scenarios, where interference stemming from neighboring cell users changes fast, which is a natural assumption in packet switch transmissions. Our target is to estimate the instantaneous or ergodic mutual information of the transmit-receive link, which serves here as an approximation of the achievable communication rate provided that no improved precoding is performed. An important usage of the mutual information estimation is found in the context of cognitive radios where multiple frequency bands are sensed for future transmissions. In this setting, the proposed estimator provides the expected rate performance (either instantaneous or ergodic) achievable in each frequency band, prior to actual transmission. The transmit-receive pair may then elect the frequency sub-bands most suitable for communication. 
	
	
        The setting of the article assumes that the channel from the transmitter to the receiver is known by the receiver (but not known by the transmitter), which is a realistic scenario provided that some channel state feedback is delivered by the transmitter, and that the statistical inference on the mutual information is based on $M$ successive observations of channel uses, where $M$ is not large compared to the number of receive antennas $N$, therefore naturally calling for the G-estimation framework. The progression of this article will consist first in studying the conventional one-step estimator, hereafter called the standard empirical (SE) estimator, which corresponds to estimating the interference covariance matrix by the empirical covariance matrix and to replacing the estimate in the mutual information formula. We then show that this approach, although consistent in the large $M$ regime, performs poorly in the regime where both $M$ and $N$ are of similar sizes. We then provide an alternative approach, based on the G-estimation scheme, and produce a novel G-estimator of the mutual information which we first prove consistent in the large $M,N$ regime and for which we derive the asymptotic second order performance through a central limit theorem.

	The remainder of the article is structured as follows. In Section \ref{sec:system_model}, the system model is described and the considered problem is mathematically formalized. In Section \ref{sec:first}, first order results for both the SE-estimator and the G-estimator are provided. In Section \ref{sec:second}, the fluctuations of the G-estimator are studied. We then provide in Section \ref{sec:simulations} numerical simulations that support the accuracy of the derived results, before concluding the article in Section \ref{sec:conclusion}. Mathematical details are provided in the appendices.

	\subsubsection*{Notations} In the following, boldface lower case symbols represent vectors, capital boldface characters denote matrices ($\I_N$ is the size-$N$ identity matrix). If ${\bf A}$ is a given matrix, ${\bf A}^{\herm}$ stands for its transconjugate; if ${\bf A}$ is square, $\tr({\bf A})$, $\det({\bf A})$ and $\|\A\|$ respectively stand for the trace, the determinant and the spectral norm of ${\bf A}$. We say that the variable $X$ has a standard complex Gaussian distribution if $X=U+\mathbf{i} V$ ($\mathbf{i}^2=-1$) , where $U,V$ are independent real random variables with Gaussian distribution ${\mathcal N}(0,1/2)$. The complex conjugate of a scalar $z$ will be denoted by $z^*$.  Almost sure convergence will be denoted by $\xrightarrow[]{\textrm{a.s.}}$, and convergence in distribution by $\xrightarrow[]{\mathcal D}$. Notation ${\mathcal O}$ will refer to Landau's notation: $u_n= {\mathcal O}(v_n)$ if there exists a bounded sequence $K_n$ such that $u_n = K_n v_n$. For a square $N\times N$ Hermitian matrix $\bf A$, we denote $\lambda_1({\bf A})\leq \ldots \leq \lambda_N({\bf A})$ the ordered eigenvalues of $\bf A$.
	 
	\section{System model and problem setting}
	\label{sec:system_model}
	
	\subsection{System model}
	Consider a wireless communication channel $\H_t\in\CC^{N\times n_0}$ between a transmitter equipped with $n_0$ antennas and a receiver equipped with $N$ antennas, the latter being exposed to interfering signals. The objective of the receiver is to evaluate the mutual information of this link during a {\it sensing period} assuming $\H_t$ known at all time. For this, we assume a block-fading scenario and denote by $T\geq 1$ the number of channel coherence intervals (or time slots) allocated for sensing. In other words, we suppose that, within each channel coherence interval $t\in\{1,\ldots,T\}$, $\H_t$ is deterministic and constant. We also denote by $M$ the number of channel uses employed for sensing during each time slot ($M$ times the channel use duration is therefore less than the channel coherence time). The $M$ concatenated signal vectors received in slot $t$ are gathered in the matrix $\overline{\Y}_t\in\CC^{N\times M}$ defined as
	$$\overline{\Y}_t=\H_t\X_{t,0}+\overline{\W}_{t}$$
	where $\X_{t,0}\in\CC^{n_0\times M}$ is the concatenated matrix of the transmitted signals and $\overline{\W}_{t}\in\CC^{N\times M}$ represents the concatenated interference vectors. 
	
	Since $\overline{\W}_{t}$ is not necessarily a white noise matrix in the present scenario, we write $\overline{\W}_{t}=\G_t\W_t$ where $\G_t\in\CC^{N\times n}$ is such that $\G_t\G_t^\herm\in\CC^{N\times N}$ is the deterministic matrix of the noise variance during slot $t$ while $\W_t\in\CC^{n\times M}$ is a matrix filled with independent entries with zero mean and unit variance. That is, we assume that the interference is stationary during the coherence time of $\H_t$, which is a reasonable assumption in practical scenarios, as commented in Remark \ref{rem:1}. The choice of using the additional system parameter $n$, not necessarily equal to $N$, is also motivated by practical applications where the sources of interference may be of different dimensionality than the number of receive antennas, as discussed in Remark \ref{rem:1} below. This will have no effect on the resulting mutual information estimators.
	
	We finally assume that perfect decoding of $\X_{t,0}$ (possibly transmitted at low rate or not transmitted at all) is achieved during the sensing period. If so, since $\H_t$ is assumed perfectly known, the residual signal to which the receiver has access is given by
	$$\Y_t=\overline{\Y}_t-\H_t\X_{t,0}=\G_t\W_{t}.$$

	\begin{remark}
		\label{rem:1}
	The usual white noise assumption naturally arises from the thermal noise created by the electronic components at the receiver radio front end as well as from the large number of exogenous sources of interference in the vicinity of the receiver. However, in cellular networks, and particularly so in cell edge conditions, the main source of interference arises from coherent transmissions in adjacent cells. In this case, only a small number $K$ of signal sources interfere in a colored manner. Calling $\G_{t,k}\in\CC^{N\times n_k}$ the channel from interferer $k\in\{1,\ldots,K\}$, equipped with $n_k$ antennas, to the receiver and $\X_{t,k}\in\CC^{n_k\times M}$ the concatenated transmit signals from interferer $k$, the received signal $\overline{\Y}_t$ can be modeled as
\begin{equation}\overline{\Y}_t=\H_t\X_{t,0}+\sum_{k=1}^K \G_{t,k} \X_{t,k}+\sigma \W'_{t}\label{eq:model_sim}\end{equation}
where $\sigma\W'_t\in \CC^{N\times M}$ is the concatenated additional white Gaussian noise with variance $\sigma^2>0$. In this case, we see that, denoting $n=n_1+\ldots+n_K+N$ and
	\begin{align*}
	\G_t&=\left[\G_{t,1},\cdots,\G_{t,K},\sigma \I_N\right] \\
	\W_t&=\left[\X_{t,1}^{\trans},\cdots,\X_{t,K}^{\trans},{\W'_t}^{\trans}\right]^{\trans}
	\end{align*}
we fall back on the above model. Figure \ref{fig:systemmodel} depicts this scenario in the case of $K=2$ interfering users.
	\end{remark}

	\begin{figure}
	\centering
	\includegraphics[width=12cm]{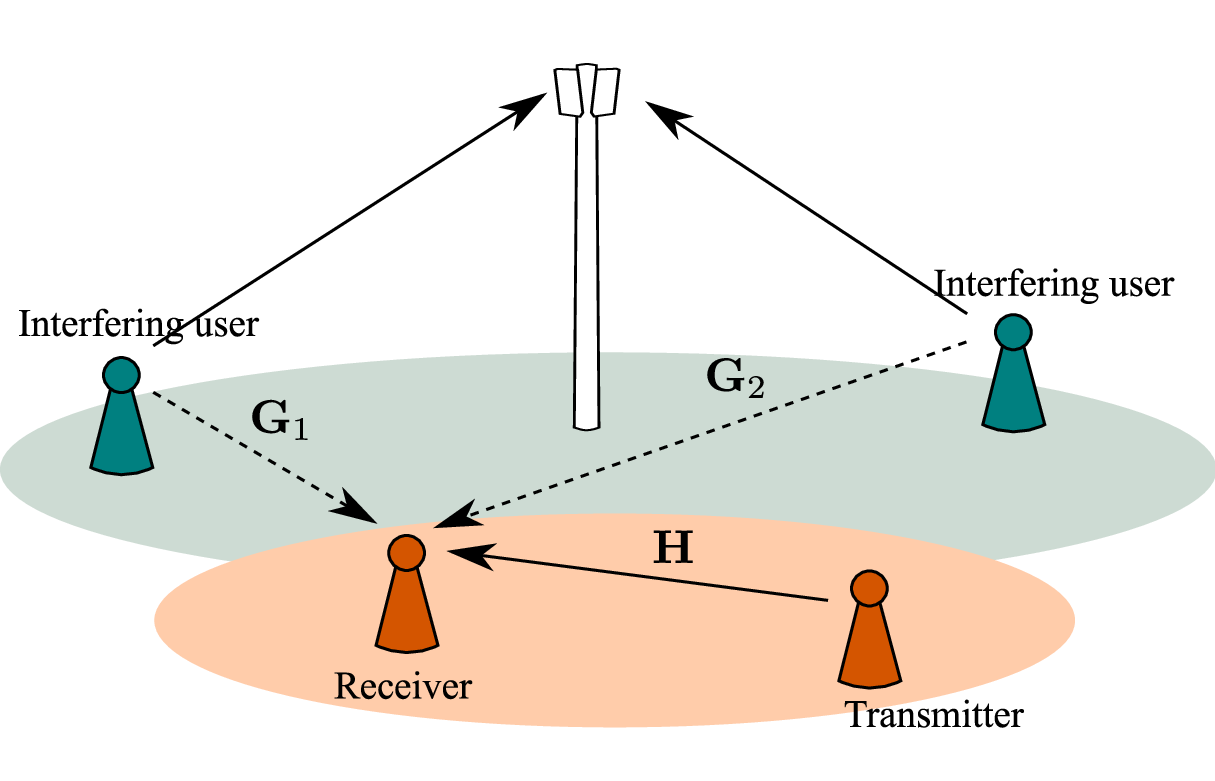}
	\caption{System model of Remark \ref{rem:1} with two interferers.}
	\label{fig:systemmodel}
 	\end{figure}

	The statistical properties of the random variables $\X_{t,0}$ and $\W_t$ are precisely described as follows.
	\begin{assumption}\label{ass:channel} 
          For a given $t$ where $1\leq t\leq T$, the entries of the matrices ${\bf X}_{t,0}$ and $\W_t$ are i.i.d. random variables with standard complex Gaussian distribution.
	\end{assumption}

	The objective for the receiver is to evaluate the {\it average (per-antenna) mutual information} that can be achieved during the $T$ slots. In particular, for $T=1$, the expression is that of the instantaneous mutual information which allows for an estimation of the rate performance of the current channel. If $T$ is large instead, this provides an approximation of the long-term ergodic mutual information. Under Assumption \ref{ass:channel}, the average mutual information is given by
	\begin{eqnarray}\label{eq:capacity}
	\erg&=&\frac{1}{NT}\sum_{t=1}^T\left[ \log\det\left(\H_t\H_t^{\herm}+\G_t\G_t^\herm\right)-\log\det\left(\G_t\G_t^\herm\right)\right].
	\end{eqnarray}
	The target of the article is to address the problem of estimating $\erg$ based on $T$ successive observations $\Y_1,\ldots,\Y_T$ assuming perfect knowledge of $\H_1,\cdots,\H_T$, but unknown $\G_t$ for all $t$.
 	
	\subsection{The standard empirical estimator $\trad$}
	If the number $M$ of available observations during the sensing period in each slot is very large compared to the channel vector $N$, a natural estimator, hereafter referred to as the standard empirical (SE) estimator, consists in the following one-step estimator
	\begin{equation}\label{eq:trad-estimator}
	\trad\quad =\quad \frac{1}{NT}\sum_{t=1}^T\log\det\left(\H_t\H_t^{\herm}+ \frac{1}{M}\Y_t\Y_t^{\herm}\right)
	-\frac{1}{NT}\sum_{t=1}^T\log\det\left(\frac{1}{M} \Y_t\Y_t^{\herm}\right)\ .
	\end{equation}
	For future use, it is convenient to introduce the notation
	\begin{eqnarray}
	\trad(y)&=& \frac{1}{NT}\sum_{t=1}^T\log\det\left(y\,\H_t\H_t^{\herm}+ \frac{1}{M}\Y_t\Y_t^{\herm}\right)
	-\frac{1}{NT}\sum_{t=1}^T\log\det\left(\frac{1}{M} \Y_t\Y_t^{\herm}\right)\ .\label{eq:def-trad-y}
	\end{eqnarray}
	With this notation at hand, $\trad=\trad(1)$.

	For $N$ fixed, it is an immediate application of the law of large numbers and of the continuous mapping theorem to observe that, as $M\to\infty$,
	\begin{equation}
		\label{eq:tradconv}
		\trad - \erg \asto 0.
	\end{equation}
	However, from the discussions above, the assumption $M\gg N$ may not be tenable for practical settings where sensing needs to be performed fast, particularly so under fast fading conditions. In this case, as will be shown in Section \ref{sec:first}, the SE-estimator is asymptotically biased in the large $M,N$ regime, hence {\em not} consistent, and \eqref{eq:tradconv} will no longer hold true. This motivates the study of an alternative consistent estimator based on the G-estimation framework. To this end, we first need to study in depth the statistical properties of the SE-estimator from which the G-estimator will naturally arise. The statistical properties of the latter will similarly be obtained by first studying the second order statistics of the SE-estimator (themselves being of limited practical interest). Before moving to our main results, we first need some further technical hypotheses.
	
	\subsection{The asymptotic regime}
	
	In this section, we formalize the conditions under which the large $M,N$ regime is considered. We will require the following assumptions.
	
\begin{assumption}\label{ass:asymptotics} $M,N,n, n_0\to+\infty$, and
	\begin{eqnarray*}
	0& <& \liminf_{M,N\rightarrow \infty} \frac Nn \quad \leq \quad \limsup_{M,N\rightarrow \infty} \frac Nn \ <\ +\infty\ ,\\
	1&<& \liminf_{M,N\rightarrow \infty} \frac MN \quad \leq\ \quad \limsup_{M,N\rightarrow \infty} \frac MN \ <\ +\infty\ ,\\
	0 &<& \liminf_{N,n_0\rightarrow \infty} \frac {n_0} N \quad \leq \quad \limsup_{N,n_0\rightarrow \infty}
	\frac {n_0} N \ <\ +\infty\ .
	\end{eqnarray*}
	\end{assumption}
	\begin{remark} 
		The constraints over $N$ and $n$ simply state that these quantities remain of the same order. The lower bound for the ratio $M/N$ accounts for the fact that that $M$ is larger than $N$, although of the same order.
	\end{remark}
	In the remainder of the article, we may refer to Assumption \ref{ass:asymptotics} as the convergence mode $M,N,n\to \infty$.
	
        We also need the channel matrices to be bounded in spectral norm, as $M,N,n\to \infty$, as follows.
\begin{assumption}\label{ass:channel-matrix-G}
Let $N=N(n)$ a sequence of integers indexed by $n$. For each $t\in\left\{1,\cdots,T\right\}$, consider the family of $N\times n$ matrices $\G_t$. Then,
\begin{itemize}
\item The spectral norms of $\G_t$ 
  are uniformly bounded in the sense that
	$$
	\sup_{1\le t\le T}\sup_{N,n}\|\G_t\|< \infty\ .
	$$
\item For $t\in \{1,\cdots,T\}$, the smallest eigenvalue of
        $\G_t \G_t^\herm$ denoted by $\lambda_N(\G_t \G_t^\herm)$ is
        uniformily bounded away from zero, i.e. there exists $\sigma^2>0$ such that
	$$
	\inf_{1\le t\le T}\inf_{N,n}\lambda_N(\G_t \G_t^\herm)  \ge \sigma^2 >0 \ .
	$$

\end{itemize}
	\end{assumption}
	\begin{assumption}\label{ass:channel-matrix-H}
Let $N=N(n_0)$ a sequence of integers indexed by $n_0$. For each $t\in\left\{1,\cdots,T\right\}$, consider the family of $N\times n_0$ matrices $\H_t$. Then,
The spectral norms of $\H_t$ are uniformly bounded in the sense that
	$$
	\sup_{1\le t\le T}\sup_{N,n_0}\|\H_t\|< \infty\ .
	$$
	\end{assumption}
	\begin{assumption}\label{ass:channel-matrix-Hb}
	The family of matrices $(\H_t)$ satisfies additionally the following assumptions:
	\begin{itemize}
	\item[1)] Denote by $p_t$ the rank of $\H_t$. Then
	$$
	0 \ <\ \liminf_{N,n_0\rightarrow\infty} \frac{p_t}N\ \leq\ \limsup_{N,n_0\rightarrow\infty} \frac{p_t}N \ <\ 1\ .
	$$
\item[2)] The smallest non-zero eigenvalue of $\H_t\H_t^{\herm}$ is uniformly bounded away from zero, i.e. there exists $\kappa>0$ such that: 
	$$
	\inf_{1\leq t\leq T}\inf_{N,n_0}\left\{\lambda_i(\H_t\H_t^{\herm}) ~\vert~ \lambda_i(\H_t\H_t^{\herm}) >0\right\} \ge \kappa >0.
	$$
\end{itemize}
\end{assumption}

%
%
%
%

	\section{Convergence of the average mutual information estimators}
	\label{sec:first}
	In this section, we study the asymptotic behavior of the SE-estimator $\trad$ and prove that under the asymptotic regime \ref{ass:asymptotics}, this estimator is asymptotically biased. Relying on this first analysis, we then derive a consistent estimator based on the random matrix inference techniques known as G-estimation.

	These techniques can be classified in two categories. One is based on the link between the Stieltjes transform (see Appendix \ref{app:stieltjes}) and the Cauchy complex integral, recently exhibited by Mestre who developed a framework for the estimation of eigenvalues and eigenspace projections \cite{MES08}. This approach is often well-adapted as long as the estimation of parameters depending either on the eigenvalues or on the eigenvector projections of ${\bf Y}_t {\bf Y}_t^\herm$ is considered (see for instance Lemma \ref{lemma:ST}) but may fail when the dependence is more involved. The second approach, which we will adopt here, is based on the technique of deterministic equivalents developed in \cite{HAC07,HAC06}. It follows from the initial work \cite{LOU09b} of Vallet and Loubaton, and will be illustrated in Section \ref{sec:g-estimator}. 

	
	
	


	\subsection{The standard empirical estimator $\trad$}

	We start by studying the second of the two terms in the difference \eqref{eq:capacity} for which it is much easier to derive an estimate.

	\begin{lemma}\label{lemma:ST}
	Let Assumptions \ref{ass:channel}-\ref{ass:channel-matrix-H} hold. Then, we have the following convergence.
	$$
	\frac{1}{N}\log\det(\G_t\G_t^{\herm})
	-\frac{1}{N}\log\det\left(\frac 1M \Y_t\Y_t^{\herm}\right)+\frac{N-M}{N}\log
	\left(\frac{M-N}{M}\right)-1\xrightarrow[M,N,n\rightarrow\infty]{\textrm{a.s.}}
	0\ .
	$$
	\label{lemma:stieltjes}
	\end{lemma}
	\begin{proof}
See Appendix \ref{app:stieltjes}
\end{proof}
\begin{remark}
It should be noted that in the proof of lemma \ref{lemma:stieltjes}, the Gaussianity assumption of the entries is not necessary and can be replaced  by a finite moment condition.
\end{remark}
\begin{remark}
	Lemma \ref{lemma:stieltjes} relies on the Stieltjes transform estimation technique from Mestre. The latter is used to compute a consistent estimate of the quantity $\frac{1}{N}\sum_{t=1}^T\log\det(\G_t\G_t^{\herm})$, which is seen here as a functional of the (non-observable) eigenvalues of $\G_t\G_t^\herm$. Following the work from Mestre \cite{MES08}, the idea is to link the Stieltjes transform of $\G_t\G_t^\herm$ to that of the (almost sure) limiting Stieltjes transform of the (observable) sample covariance matrix $\frac{1}{M}\Y_t\Y_t^{\herm}$. See \cite{COU12} for a tutorial on these notions.
\end{remark}

As a consequence of Lemma \ref{lemma:ST}, we see that the $\frac{1}{N}\log\det(\frac 1M \Y_t\Y_t^{\herm})$ is a consistent estimate of $\frac{1}{N}\log\det(\G_t\G_t^{\herm})$ (recall that $\frac{1}{M} \mathbb{E} {\Y}_t {\Y}_t^{\herm} = \G_t \G_t^{\herm}$) up to a bias term depending on the time and space dimensions only. This may suggest that, up to the introduction of the term $\H_t\H_t^{\herm}$ in the log determinants for estimating the first term in \eqref{eq:capacity}, the SE-estimator is also a consistent estimator for $\erg$. This is however not true. To study the first term in \eqref{eq:capacity}, which is not as immediate as the second term, we need some further work. We start with a first technical lemma which follows instead from random matrix operations on deterministic equivalents. \footnote{By deterministic equivalents, we mean deterministic quantities which are asymptotically close to the quantity under investigation. The advantage of considering such equivalents comes from the fact that this prevents from studying the true limit of the quantities under investigation (which might not exist anyway). See \cite{HAC07} for more details.}
	
	\begin{lemma}
	Let Assumptions \ref{ass:channel}--\ref{ass:channel-matrix-H} hold and let $y>0$. Then we have the following identities.
	\begin{enumerate}
	\item The fixed-point equation in $y$
	\begin{equation}
	\kappa_t(y)=\frac{1}{M}\tr\left(\G_t \G_t^\herm \left(\frac{\G_t \G_t^\herm}{1+ \kappa_t(y)}+y\H_t\H_t^{\herm}\right)^{-1}\right)
	\label{eq:kappa_t_t}
	\end{equation}
	admits a unique positive solution $\kappa_t(y)$.
	\end{enumerate}
	Denote by $\T_t(y)$ and $\Q_t(y)$ the following quantities:
	$$
	\T_t(y)=\left(y\H_t\H_t^{\herm}+\frac{\G_t \G_t^\herm}{1+\kappa_t(y)}\right)^{-1}\ ,\quad
	\Q_t(y)=\left(y\H_t\H_t^{\herm}+\frac{1}{M}\Y_t\Y_t^{\herm}\right)^{-1}\ .
	$$
	\begin{enumerate}
	\item[2)] Then, for any deterministic family $({\bf S}_N)$ of $N\times
	N$ complex matrices with uniformly bounded spectral norm, we have:
	$$
	\frac{1}{M}\tr{\bf S}_N\Q_t(y)-\frac{1}{M}\tr{\bf S}_N{\T}_t(y)\xrightarrow[M,N,n\rightarrow\infty]{\textrm{a.s.}} 0.
	$$
	\item[3)] Let
	$$
	V_t(y)=\log\det\left(y\H_t\H_t^{\herm}+\frac{\G_t \G_t^\herm}{1+\kappa_t(y)}\right)+M\log(1+\kappa_t(y))-M\frac{\kappa_t(y)}{1+\kappa_t(y)}.
	$$
	Then, the following convergence holds
	$$
	\frac{1}{N}\log\det\left(y\H_t\H_t^{\herm}+\frac{1}{M}\Y_t\Y_t^{\herm}\right)-\frac{1}{N}V_t(y)\xrightarrow[M,N,n\rightarrow\infty]{\textrm{a.s.}} 0\ .
	$$
	\end{enumerate}
	\label{lemma:deter}
	\end{lemma}
	\begin{proof}	
	See Appendix \ref{app:lemma_deter}.
\end{proof}
	
Clearly, when setting $y=1$, this result provides a convergence result for the SE-estimator, as will be stated in Theorem \ref{th:bias}. Lemma \ref{lemma:deter} is however more generic in its replacing the term $1$ in front of $\H_t\H_t^\herm$ by an auxiliary parameter $y$. As a matter of fact, the introduction of $y$ is at the core of the novel estimator derived later. We can indeed already anticipate the remainder of the derivations: if $(1+\kappa_t(y))^{-1}$ can be made equal to $y$, then the first term in the expression of $V_t(y)$ is proportional to the first term in \eqref{eq:capacity} which we are interested in. Turning the factor $1$ into a generic variable $y$ will therefore provide the flexibility missing to estimate \eqref{eq:capacity} precisely in the large $M,N,n$ regime. Before getting into these considerations, let us start with the following result on the SE-estimator. 

	
	\begin{theorem}[Asymptotic bias of the SE-estimator]
	Let Assumptions \ref{ass:channel}-\ref{ass:channel-matrix-H} hold, and denote 
	\begin{align}\label{eq:def-V}
	\mathcal{V}(y)&=\frac{1}{NT}\sum_{t=1}^T \left( \log\det\left( y\H_t\H_t^{\herm}+\frac{\G_t\G_t^\herm}{1+\kappa_t(y)}\right)-\log\det(\G_t\G_t^{\herm})\right)\nonumber \\
	&+ \frac{1}{T}\sum_{t=1}^T\left(\frac MN \log(1+\kappa_t(y))-\frac MN \frac{\kappa_t(y)}{1+\kappa_t(y)}\right)+\frac{M-N}N \log\left(\frac{M-N}{M}\right)+1.
	\end{align}
	where $\kappa_t(y)$ is the unique solution of \eqref{eq:kappa_t_t}. Then,
	$$
	\trad-\mathcal{V}(1)\xrightarrow[M,N,n\rightarrow\infty]{\textrm{a.s.}} 0\ .
	$$
	\label{th:bias}
	\end{theorem}
	\begin{proof}
	Gathering item 3) of Lemma \ref{lemma:deter} together with Lemma \ref{lemma:stieltjes} yields the desired result.
	\end{proof}

	This result suggests that the SE-estimator is not necessarily a consistent estimator of the mutual information, as there is no reason for the bias term in \eqref{eq:def-V} (for $y=1$) to be identically null. However, based on the discussion prior to Theorem \ref{th:bias}, we are now in position to derive a novel consistent estimator. The following section is dedicated to this task.

	\subsection{A G-estimator of the average mutual information}\label{sec:g-estimator}
	
	The following result is our main contribution, which provides the novel consistent estimator for \eqref{eq:capacity}.

	\begin{theorem}[G-estimator for the average mutual information]
	Assume that \ref{ass:channel}-\ref{ass:channel-matrix-H} hold and define the quantity
	\begin{align*}
	\gest &= \frac{1}{NT}\sum_{t=1}^T\log\det\left(\I_N+\yest\H_t\H_t^{\herm}\left(\frac{1}{M}\Y_t\Y_t^{\herm}\right)^{-1}\right)\\
	&+\frac{1}{T}\sum_{t=1}^T\frac{(M-N)}{N}\left[\log\left(\frac{M}{M-N}\yest \right)+1\right]-\frac{M}{N}\yest
	\end{align*}
	where $\yest$ is the unique real positive solution of
	$$
	\yest =\frac{\yest}{M}\tr
	\H_t\H_t^{\herm}\left(\yest\,\H_t\H_t^{\herm}+\frac{1}{M}\Y_t\Y_t^{\herm}\right)^{-1}+\frac{M-N}{M}.
	$$
	Then
	$$
	\gest-\erg\xrightarrow[M,N,n\rightarrow\infty]{\textrm{a.s.}} 0\ .
	$$
	\label{th:gestimator}
	\end{theorem}
	\begin{proof}
We hereafter provide an outline of the proof, which is developed in full detail in Appendix \ref{app:g_estimator}. Denote ${\mathcal I}_t$ the average mutual information at time $t$ as
	\begin{eqnarray*}
	{\mathcal I}_t&\triangleq&\frac{1}{N}\log\det(\G_t\G_t^{\herm}+\H_t\H_t^{\herm})
	-\frac 1N \log\det(\G_t\G_t^{\herm})\ ,\\
	&\triangleq& {\mathcal I}_{t,1}-{\mathcal I}_{t,2}.
	\end{eqnarray*}
	Recall that a consistent estimate $\hat{\mathcal I}_{t,2}$ of ${\mathcal I}_{t,2}$ was provided in Lemma \ref{lemma:stieltjes}. It therefore remains to build a consistent estimate for ${\mathcal I}_{t,1}$.
	
	The proof is divided into four steps, as follows.
	\begin{enumerate}
	\item In the first step, we exploit the convergence of parametrized quantities of interest. Denote
	$$f(y)=\frac{1}{N}\log\det\left(\frac{1}{M}\Y_t\Y_t^{\herm}+y\H_t\H_t^{\herm}\right)$$ and recall the
	definition of $\kappa_t(y)$ as given in Lemma
	\ref{lemma:deter}-1). By Lemma \ref{lemma:deter}-3),
	$$
	-f(y)+\frac{1}{N}\log\det\left(\frac{\G_t\G_t^{\herm}}{1+\kappa_t(y)}+y\H_t\H_t^{\herm}\right)
	+\frac{M}{N}\log(1+\kappa_t(y))-\frac{M}{N}\frac{\kappa_t(y)}{1+\kappa_t(y)}
	\xrightarrow[M,N,n\rightarrow\infty]{\mathrm{a.s}} 0\ .
	$$
	Clearly, for most values of $y$; the deterministic quantity to which $f(y)$ converges differs from $\II_{t,1}$.
	
	\item In the second step, we find a specific value of $y$ to enforce
	the desired quantity $\II_{t,1}$ to appear. One can readily check
	that if $y_{N,t}$ is the solution of the equation in $y$
	\begin{equation}\label{eq:y-def}
	y=\frac 1{1+\kappa_t(y)}
	\end{equation}
	then we immediately obtain
	\begin{equation}\label{eq:C1-est-y}
	\II_{t,1} -\left[\frac{1}{N}\log\det\left(\frac{1}{M}\Y_t\Y_t^{\herm}+y_{N,t}\H_t\H_t^{\herm}\right)+\frac{M-N}{N}\log(y_{N,t})+\frac{M}{N}(1-y_{N,t})\right]\xrightarrow[M,N,n\rightarrow\infty]{\textrm{a.s.}}
	0\ .
	\end{equation}
	From the definition of $\kappa_t(y)$, we show that there
	exists a unique positive $y_{N,t}$ solution of \eqref{eq:y-def}, given by the closed-form
	expression
	\begin{equation}\label{eq:y-exp}
	y_{N,t}=1-\frac{1}{M}\tr\left[(\G_t\G_t^{\herm})(\H_t\H_t^{\herm}+\G_t\G_t^{\herm})^{-1}\right]\ .
	\end{equation}
	However, the value of $y_{N,t}$ still depends upon the unknown matrix $\G_t$ to this point.
	
	\item In the third step, we provide a consistent estimator $\yest$ of
	$y_{N,t}$. Based on an analysis of $\kappa_t(y)$, and on finding a
	consistent estimate for this quantity, we show that there exists
	a unique positive solution $\yest$ to 
	\begin{equation}
	\yest=\frac{1}{M}\tr \yest\, \H_t \H_t^{\herm} \left(\yest\, \H_t \H_t^{\herm}+\frac{1}{M}\Y_t\Y_t^{\herm}\right)^{-1}+\frac{M-N}{M}.
	\label{eq:yest}
	\end{equation}
	Moreover, $\yest$ satisfies
	$$
	\yest - y_{N,t} \xrightarrow[M,N,n\to \infty]{\mathrm{a.s.}} 0\ .
	$$
	
	\item Finally, it remains to check that we can replace $y_{N,t}$ by
	$\yest$ in the convergence \eqref{eq:C1-est-y}. This immediately
	yields a consistent estimate $\hat{\II}_{t,1}$ for $\II_{t,1}$. For the proof of the
	theorem to be complete, it remains to gather the estimates of $\II_{t,1}$ and $\II_{t,2}$, which finally yields the announced result
	$$
	\hat{\II}_G\ =\ \frac 1T \sum_{t=1}^T \left( \hat{\II}_{t,1} - \hat{\II}_{t,2}\right).
	$$
	\end{enumerate}
	\end{proof}

	\section{Fluctuations of the G-estimator}
	\label{sec:second}
	
	In this section, we establish a central limit theorem for the improved G-estimator $\gest$, so to evaluate the asymptotic performance of our novel estimator. Due to the Gaussian assumption on $\W_t$, we can use the powerful {\it Gaussian methods} developed for the study of large random matrices by Pastur et al. \cite{PAS98,HAC06}. In order to derive the asymptotic fluctuations of the G-estimator $\gest$, similar to the previous section, a first step consists in evaluating the fluctuations of $\trad(y)$.
	
	\begin{theorem}
	Let Assumptions \ref{ass:channel}-\ref{ass:channel-matrix-Hb} hold and recall the definition \eqref{eq:def-trad-y} of $\trad(y)$. We then have the following results.
	\begin{enumerate}
	\item The sequence of real numbers
	\begin{align*}
	\alpha_N(y)&=\frac{2\log(M)}{T^2}-\frac{1}{T^2}\sum_{t=1}^T\log\left[(M-N)\Bigg(M(\kappa_t(y)+1)^2 
	-\tr\left(\frac{\I_N}{\kappa_t(y)+1}+y\H_t\H_t^{\herm}(\G_t\G_t^{\herm})^{-1}\right)^{-2}\Bigg)\right]
	\end{align*}
	is well-defined and
	$$
	0\ <\ \liminf_{M,N,n\rightarrow\infty} \alpha_N(y)\ \leq\ \limsup_{M,N,n\rightarrow\infty} \alpha_N(y)\ <\ +\infty\ .
	$$
	\item The following convergence holds 
	$$
	\frac{N}{\sqrt{\alpha_N(y)}}\left( \trad(y) - \mathcal{V}(y)
	\right)\xrightarrow[N,M,n\to\infty]{\mathcal{D}}\mathcal{N}(0,1)
	$$
	where $\mathcal{V}(y)$ is defined in \eqref{eq:def-V}.
	\end{enumerate}
	\label{th:clt_main}
	\end{theorem}
	\begin{proof}
	See Appendix \ref{app:clt_main}.
	\end{proof}

	\label{sec:performance_g_estimator}
	With the above result at hand, we are now in position to derive the fluctuations of the G-estimator. As opposed to $\trad(y)$ though, the G-estimator has no closed-form expression, as the $\yest$'s are solutions of implicit equations. Establishing a CLT for $\gest$ therefore requires to control both the fluctuations of the received matrix $\Y_t$ and of the quantity $\yest$. In the following lemma, we first prove that the fluctuations of $\yest-y_{N,t}$ are of order $\mathcal{O}(M^{-2})$, a rate which will turn out to be sufficiently fast to discard the randomness stemming from $\yest$ in the asymptotic fluctuations of $\gest$.
	
	\begin{lemma}
	For $t\in \{1,\cdots,T\}$, the following estimates hold true, as $M,N,n\to \infty$:
	\begin{enumerate}
	\item ${\rm var}(\yest)=\mathcal{O}(M^{-2})$ ,
	\item $\EE\, \yest= y_{N,t}+\mathcal{O}(M^{-2})$ .
	\end{enumerate}
	\label{th:y}
	\end{lemma}
	\begin{proof}
	See Appendix \ref{app:proofy}.
	\end{proof}
	
	We are now in position to state the central limit theorem for $\gest$.
	\begin{theorem}
	Let Assumptions \ref{ass:channel}-\ref{ass:channel-matrix-Hb} hold true. Then,
	$$
	\frac{N}{\sqrt{\theta_N}}(\gest-\erg)\xrightarrow[N\to\infty]{\mathcal{D}}\mathcal{N}(0,1)
	$$
	where $\theta_N$ given by
	\begin{equation}
	\theta_N=\frac{1}{T^2}\sum_{t=1}^T 2\log(My_{N,t})-\log\left[(M-N)\left(M-\tr\left(\I_N+\H_t\H_t^{\herm}(\G_t\G_t^{\herm})^{-1}\right)^{-2}\right)\right]
	\label{eq:theta}
	\end{equation}
	which is a well-defined quantity which satisfies
	$$
	0\ <\ \liminf_{M,N,n\rightarrow\infty} \theta_N\ \leq\ \limsup_{M,N,n\rightarrow\infty} \theta_N\ <\ +\infty\ .
	$$
	\label{th:clt_gestimator}
	\end{theorem}
	\begin{proof} Consider the function $\mathcal{J}_t(y)$ defined for $y>0$
	as:
	\begin{multline*}
	\mathcal{J}_t(y)=\frac{1}{N}\log\det\left(y\H_t\H_t^{\herm}+\frac{\Y_t\Y_t^{\herm}}{M}\right)
	+\frac{M-N}{N}\left[\log\left(\frac{M}{M-N}y\right)+1\right]-\frac{M}{N}y
	-\log\det\left(\frac{\Y_t\Y_t^{\herm}}{M}\right)\
	.
	\end{multline*}
	Then $\gest=\frac 1T \sum_{t=1}^T {\mathcal J}_t(\hat{y}_{N,t})$. Since
	all the random variables $({\mathcal J}_t(\hat{y}_{N,t}),\ 1\le t\le T)$
	are independent, it is sufficient to prove a CLT for
	${\mathcal J}_t(\hat{y}_{N,t}) $, for a given $t\in \{1,\cdots, T\}$. In order to
	handle the randomness of $\hat{y}_{N,t}$, we shall perform a Taylor
	expansion of ${\mathcal J}_t$ around $\hat{y}_{N,t}$. Recall the following differentiation formula
	$$
	\frac d{dx} \log \det A(x) = \tr A'(x) A^{-1}(x).
	$$
	A direct application of
	this formula, together with the mere definition of $\hat{y}_{N,t}$ yields
	$$
	\frac{d\, {\mathcal J}_t}{d\, y} \left(\hat{y}_{N,t}\right) = 0\ .
	$$
	Hence, the Taylor expansion writes:
	\begin{equation}
	N\mathcal{J}_t(y_{N,t})
	=N\mathcal{J}_t(\hat{y}_{N,t})+N\frac{(y_{N,t}-\hat{y}_{N,t})^2}2
	\frac{d^2 \mathcal{J}_t}{d y^2}(\hat{y}_{N,t})+
	N\,\frac{(y_{N,t}-\hat{y}_{N,t})^3}{6} \frac{d^3
	\mathcal{J}_t}{d
	y^3}(\xi_{N,t})\ ,
	\label{eq:taylor}
	\end{equation}
	where $\xi_{N,t}$ lies between $y_{N,t}$ and $\hat{y}_{N,t}$. The definition \eqref{eq:yest} of $\yest$ yields
	$$
	\frac{M-N}{M} \leq \yest \leq 1+\frac{M-N}{M}\ .
	$$
	In particular, $\yest$ uniformly belongs to a fixed compact
        interval, and so does $y_{N,t}$ for similar reasons. One can
        easily prove that the second and third derivatives of
        $\mathcal{J}_t(y)$ are uniformly bounded on the union of these
        intervals.  This result combined with the fact that
        $N\EE(\yest-y_{N,t})^2=\mathcal{O}(M^{-1})$ implies that the
        last two terms in the right hand side (r.h.s.) of
        \eqref{eq:taylor} converge to zero in probability.
	By Slutsky's lemma \cite{VAN00}, it suffices to establish
        the CLT for $N\mathcal{J}(y_{N,t})$ instead of
        $N\mathcal{J}(\yest)=N\hat{\mathcal J}_G$. This is extremely helpful
        since unlike $\yest$ which is random, $y_{N,t}$ is
        deterministic. The result is thus obtained by applying Theorem
        \ref{th:clt_main} and noticing that
        $\kappa(y_{N,t})+1=\frac{1}{y_{N,t}}$. Note that although
        being valid only for fixed $y$, Theorem \ref{th:clt_main}
        could be applied by considering the slightly different model
        $\widetilde{\bf H}_t=\sqrt{{y}_{N,t}}\H_t$.
	\end{proof}

	\section{Simulations}

	In the simulations, we consider the case where a mobile terminal with $N=4$ antennas receives during a sensing period of $T$ slots data stemming from an $n_0=4$ antenna secondary transmitter. We also set the number of symbols for sensing per slot to $M=15$.  We assume that the communication link is degraded by both additive white Gaussian noise with covariance $\sigma^2\I_N$  and interference caused by $K=8$ mono-antenna users.  Hence, this scenario follows the model described by \eqref{eq:model_sim}, where for each $t$, the vectors $\G_{t,k},~k\in\left\{1,\cdots,8\right\}$ respectively represent the channel from the interferers to the receiver, whereas $\H_t$ represent the channel with the transmitter. Denote by ${\bf B}_t=\left[\G_{t,1},\cdots,\G_{t,8}\right]$. In the simulations, $\H_t$ and $\B_t$ are randomly chosen as Gaussian matrices and remain constant during the Monte Carlo averaging. To control the interference level, we scale the matrix $\B_t$ for each $t$ so that the signal-to-interference ratio SIR be given by
	\begin{align*}
		{\rm SIR}_t\triangleq\frac{\tr\H_t\H_t^{\herm}}{\tr\B_t\B_t^{\herm}} = \alpha.
	\end{align*}
	In a first experiment we set $T=10$ and ${\rm SNR}=\frac{1}{\sigma^2}=10~{\rm dB}$ and represent in Figure \ref{fig:theoretical_empirical} the theoretical and empirical normalized mean square errors for the G-estimator with respect to the SIR given respectively by:
	\begin{align*}
		{\rm MSE}_{\rm th}&=\frac{\theta_N}{\erg^2}\ ,\\
		{\rm MSE}_{\rm g,emp}&=\frac{1}{P}\sum_{i=1}^{P}\frac{N^2(\gest^{i}-\erg)^2}{\erg^2}\ ,
	\end{align*}
	where $\gest^{i}$ is the G-estimator at the $i$-th Monte Carlo iteration and $P=10\,000$ is the total number of iterations. We also display in the same graph the empirical normalized mean square error of the SE-estimator defined as
	\begin{align*}
		{\rm MSE}_{\rm t,emp}=\frac{1}{P}\sum_{i=1}^{P}\frac{N^2(\trad^{i}-\erg)^2}{\erg^2}.
	\end{align*}
	We observe that the G-estimator exhibits better performance for the whole SIR range. These results are somewhat in contradiction with the intuition that a low level of interference tends to have a small impact on the accuracy of the SE-estimator. The reason is that the mutual information depends rather on the inverse of the covariance of the interference and noise signals $\B_t\B_t^{\herm}+\sigma^2\I_N$, as
	\begin{align*}
	\log\det(\H_t\H_t^{\herm}+\B_t\B_t^{\herm}+\sigma^{2}\I_N)-\log\det(\B_t\B_t^{\herm}+\sigma^{2}\I_N)=\log\det(\H_t\H_t^{\herm}(\B_t\B_t^{\herm}+\sigma^{2}\I_N)^{-1}+\I_N).
\end{align*}
		
	We study in a second experiment the effect of $T$ when the SNR and the SIR are set respectively to $10$ dB and $-10$ dB. Figure \ref{fig:theoretical_empirical_T} depicts the obtained results. We observe that, since the SE-estimator is asymptotically biased, its mean square error does not significantly decrease with $T$ and remains almost unchanged, whereas the G-estimator exhibits a low variance which drops linearly with $T$. Finally, to assess the Gaussian behavior of the proposed estimator, we represent in Figure \ref{fig:histogram_gestimator} its corresponding histogram. We note a good fit between theoretical and empirical results although the system dimensions are small.
	
	\label{sec:simulations}
	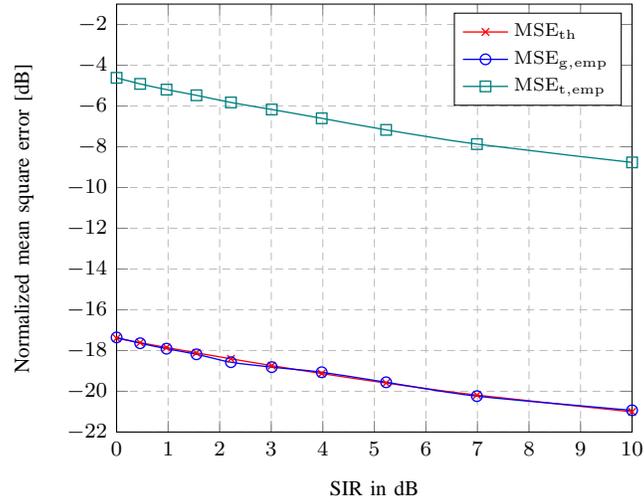
\begin{figure}[htbp]
	\centering
	\begin{tikzpicture}[font=\footnotesize]
	\renewcommand{\axisdefaulttryminticks}{8}
	\tikzstyle{every major grid}+=[style=densely dashed]
	\tikzstyle{every axis legend}+=[cells={anchor=west},fill=white,
	at={(0.98,0.98)}, anchor=north east, font=\scriptsize ]
	\begin{axis}[
	grid=major,
	xmajorgrids=true,
	ymajorgrids=true,
	xlabel={SIR in dB},
	ylabel={Normalized mean square error [dB]},
	xmin=0,
	xmax=10,
	ymin=-22,
	ymax=-1
	]
	\addplot[smooth,red,line width=0.5pt,mark=x] plot coordinates{ (10.000000,-21.005664)(6.989700,-20.186123)(5.228787,-19.595643)(3.979400,-19.128345)(3.010300,-18.739446)(2.218487,-18.405344)(1.549020,-18.111934)(0.969100,-17.850064)(0.457575,-17.613429)(-0.000000,-17.397492) };
	\addplot[smooth,blue,line width=0.5pt,mark=o] plot coordinates{(10.000000,-20.932585)(6.989700,-20.237732)(5.228787,-19.557576)(3.979400,-19.064755)(3.010300,-18.814253)(2.218487,-18.572156)(1.549020,-18.184464)(0.969100,-17.908802)(0.457575,-17.641658)(-0.000000,-17.355369) };
 	\addplot[smooth,teal,line width=0.5pt,mark=square] plot coordinates{(10.000000,-8.767107)(6.989700,-7.869041)(5.228787,-7.171355)(3.979400,-6.603648)(3.010300,-6.171194)(2.218487,-5.825293)(1.549020,-5.473996)(0.969100,-5.194771)(0.457575,-4.914932)(-0.000000,-4.617172)};
	\legend{ ${\rm MSE}_{\rm th}$, ${\rm MSE}_{\rm g,emp}$, ${\rm MSE}_{\rm t,emp}$}
 	\end{axis}
 	\end{tikzpicture}
 	\caption{Empirical and theoretical variances with respect to the SIR.}
 	\label{fig:theoretical_empirical}
 	\end{figure}
 	
 	\begin{figure}[htbp]
 	\centering
 	\begin{tikzpicture}[font=\footnotesize]
 	\renewcommand{\axisdefaulttryminticks}{8}
 	\tikzstyle{every major grid}+=[style=densely dashed]
 	\tikzstyle{every axis legend}+=[cells={anchor=west},fill=white,
 	at={(0.98,0.98)}, anchor=north east, font=\scriptsize ]
 	\begin{semilogxaxis}[
 	grid=major,
 	xmajorgrids=true,
 	ymajorgrids=true,
 	xlabel={$T$},
 	ylabel={Normalized mean square error [dB]},
 	xmin=10,
 	xmax=100,
 	ymin=-32,
 	ymax=0
 	]
 	\addplot[smooth,red,line width=0.5pt,mark=x] plot coordinates{(10.000000,-20.856221)(20.000000,-24.016069)(30.000000,-25.712690)(40.000000,-27.001156)(50.000000,-27.955606)(60.000000,-28.679551)(70.000000,-29.398445)(80.000000,-30.010374)(90.000000,-30.424841)(100.000000,-30.903397) };
 	\addplot[smooth,blue,line width=0.5pt,mark=o] plot coordinates{(10.000000,-20.816409)(20.000000,-24.062444)(30.000000,-25.708412)(40.000000,-26.967560)(50.000000,-27.945496)(60.000000,-28.664657)(70.000000,-29.475521)(80.000000,-29.960680)(90.000000,-30.473264)(100.000000,-30.855499) 	};
 	\addplot[smooth,teal,line width=0.5pt,mark=square] plot coordinates{
 (10.000000,-8.556748)(20.000000,-8.894696)(30.000000,-8.887953)(40.000000,-8.922433)(50.000000,-8.961933)(60.000000,-8.892286)(70.000000,-8.910606)(80.000000,-8.984808)(90.000000,-8.874300)(100.000000,-8.924327)
 	};
	\legend{ ${\rm MSE}_{\rm th}$, ${\rm MSE}_{\rm g,emp}$, ${\rm MSE}_{\rm t,emp}$}
 	\end{semilogxaxis}
 	\end{tikzpicture}
 	\caption{Empirical and theoretical variances with respect to $T$.}
 	\label{fig:theoretical_empirical_T}
 	\end{figure}
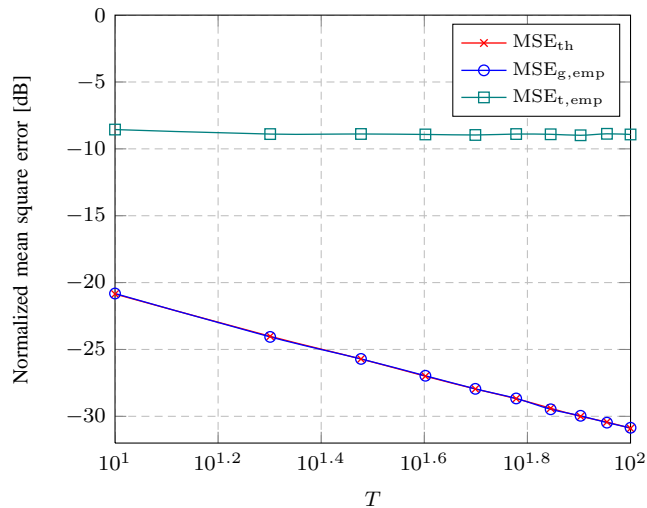

 	\begin{figure}
 	\centering
 	\begin{tikzpicture}[font=\footnotesize]
 	\renewcommand{\axisdefaulttryminticks}{4}
 	
 	\begin{axis}[
 	xmin=-4,
 	ymin=0,
 	xmax=4,
 	ymax=0.5,
 	bar width=4pt,
 	grid=major,
 	ymajorgrids=false,
 	scaled ticks=true,
 	xlabel={$\frac{N}{\sqrt{\theta_N}}(\gest-\erg)$},
 	ylabel={Frequency of occurence},
 	]
 	\addplot+[ybar,mark=none,color=black,fill=blue!40!white] coordinates{
 (-5.000000,0.000000)(-4.800000,0.000000)(-4.600000,0.000000)(-4.400000,0.000000)(-4.200000,0.000000)(-4.000000,0.000000)(-3.800000,0.000000)(-3.600000,0.000000)(-3.400000,0.000000)(-3.200000,0.001000)(-3.000000,0.005000)(-2.800000,0.006000)(-2.600000,0.014000)(-2.400000,0.018000)(-2.200000,0.030000)(-2.000000,0.031000)(-1.800000,0.084000)(-1.600000,0.102000)(-1.400000,0.177000)(-1.200000,0.207000)(-1.000000,0.260000)(-0.800000,0.329000)(-0.600000,0.356000)(-0.400000,0.378000)(-0.200000,0.408000)(0.000000,0.415000)(0.200000,0.366000)(0.400000,0.372000)(0.600000,0.316000)(0.800000,0.264000)(1.000000,0.223000)(1.200000,0.164000)(1.400000,0.149000)(1.600000,0.095000)(1.800000,0.078000)(2.000000,0.051000)(2.200000,0.039000)(2.400000,0.026000)(2.600000,0.018000)(2.800000,0.007000)(3.000000,0.004000)(3.200000,0.006000)(3.400000,0.001000)(3.600000,0.000000)(3.800000,0.000000)(4.000000,0.000000)(4.200000,0.000000)(4.400000,0.000000)(4.600000,0.000000)(4.800000,0.000000)(5.000000,0.000000) 	};
 	\addplot[smooth,red,line width=0.5pt] plot coordinates{
 	(-5.000000,0.000001) (-4.800000,0.000004) (-4.600000,0.000010) (-4.400000,0.000025) (-4.200000,0.000059) (-4.000000,0.000134) (-3.800000,0.000292) (-3.600000,0.000612) (-3.400000,0.001232) (-3.200000,0.002384) (-3.000000,0.004432) (-2.800000,0.007915) (-2.600000,0.013583) (-2.400000,0.022395) (-2.200000,0.035475) (-2.000000,0.053991) (-1.800000,0.078950) (-1.600000,0.110921) (-1.400000,0.149727) (-1.200000,0.194186) (-1.000000,0.241971) (-0.800000,0.289692) (-0.600000,0.333225) (-0.400000,0.368270) (-0.200000,0.391043) (0.000000,0.398942) (0.200000,0.391043) (0.400000,0.368270) (0.600000,0.333225) (0.800000,0.289692) (1.000000,0.241971) (1.200000,0.194186) (1.400000,0.149727) (1.600000,0.110921) (1.800000,0.078950) (2.000000,0.053991) (2.200000,0.035475) (2.400000,0.022395) (2.600000,0.013583) (2.800000,0.007915) (3.000000,0.004432) (3.200000,0.002384) (3.400000,0.001232) (3.600000,0.000612) (3.800000,0.000292) (4.000000,0.000134) (4.200000,0.000059) (4.400000,0.000025) (4.600000,0.000010) (4.800000,0.000004) (5.000000,0.000001)
 	};
 	\legend{{Histogram},{Theory}}
 	\end{axis}
 	\end{tikzpicture}
 	\caption{Histogram of $\frac{N}{\sqrt{\theta_N}}(\gest-\erg)$.}
 	\label{fig:histogram_gestimator}
 	\end{figure}
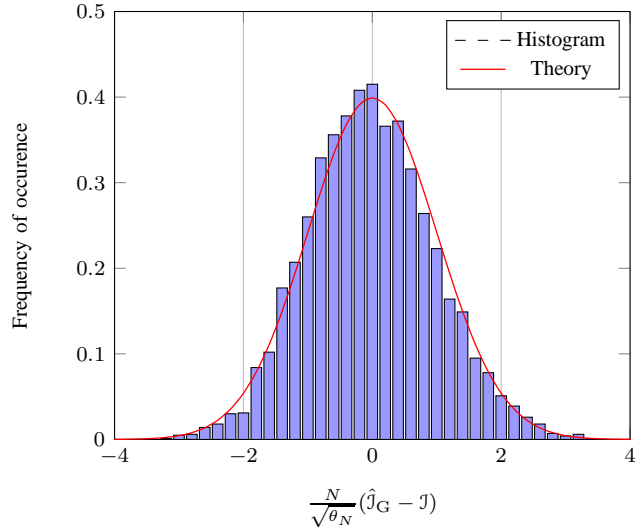
 	\section{Conclusion}
	\label{sec:conclusion}
 	In this paper, we have proposed a novel G-estimator for fast estimation of the MIMO mutual information in the presence of unknown interference in the case where the number of available observations is of the same order as the number of receive antennas. Based on large random matrix theory, we have proved that the G-estimator is asymptotically unbiased and consistent, and have studied its fluctuations. Numerical simulations have been provided and strongly support the accuracy of our results even for usual system dimensions.

 	\section*{Acknowledgment}
 	The authors would like to thank Jakob Hoydis for useful discussions.
 	\appendices
 	
 	\section{Proof of Lemma \ref{lemma:stieltjes}}
 	\label{app:stieltjes}

        Recall that if $\mathbb{P}$ is a probability distribution on
        $\mathbb{R}^+$, then the Stieltjes transform $m(z)$ of
        $\mathbb{P}$ is defined as
	\begin{equation}\label{eq:stieltjes}
	m(z)=\int_{\mathbb{R}} \frac{\mathbb{P}(d\lambda)}{\lambda-z}\ ,\quad z\in \mathbb{C} \setminus \mathbb{R}^+\ .
	\end{equation}
	For example, the Stieltjes transform $m_{\frac 1M \Y_t\Y_t^{\herm}}$ associated to the empirical
	distribution of the eigenvalues of the Hermitian matrix
	$\frac 1M \Y_t\Y_t^{\herm}$ is simply the normalized trace of the associated
	resolvent:
	$$
	m_{\frac 1M \Y_t\Y_t^{\herm}}(z) 
	=\frac{1}{N}\sum_{i=1}^N \frac{1}{\lambda_i-z}= \frac 1N \tr \left(\frac 1M \Y_t\Y_t^{\herm} - z{\bf I}_N\right)^{-1}\ ,
	$$
	where $\lambda_1,\cdots,\lambda_N$ denotes the eigenvalues of $\frac 1M \Y_t\Y_t^{\herm}$.
	Since their introduction by Mar\v{c}enko and Pastur in their seminal
	paper \cite{MAR67}, Stieltjes transforms have proved to be a highly
	efficient tool to study the spectrum of large random matrices. From
	an estimation point of view, Stieltjes transform are, in the large
	dimension regime of interest, consistent estimates of well-identified
	deterministic quantities. Therefore, the approach below consists in
	expressing the parameters of interest as functions of the
	Stieltjes transform of the eigenvalue distribution of
	$\frac 1M \Y_t\Y_t^{\herm}$.

 	Using the same eigenvalue decomposition as in Appendix
        \ref{app:lemma_deter}, we can prove that
        $\Y_t=\U_t\D_t^{\frac{1}{2}}\widetilde{\W}_t$ where
        $\widetilde{\W}_t$ is an $N\times M$ standard Gaussian matrix,
        and where $\D_t$ is a diagonal matrix with the same
        eigenvalues as $\G_t \G_t^{\herm}$. In the sequel, if ${\bf
          A}$ is a $p\times p$ hermitian matrix, denote by $F^{\bf A}$
        the empirical distribution of its eigenvalues, i.e. $F^{\bf A}
        = \frac 1p \sum_{i=1}^p \delta_{\lambda_i({\bf A})}$, and by
        $m_{\bf A}$ the associated Stieltjes transform.
 	
        Notice that due to Assumption \ref{ass:channel-matrix-G}, the following decomposition holds true:
$$
\G_t \G_t^\herm = \sigma^2 \I_N + \GG_t ,
$$
where $\GG_t$ is a positive semi-definite matrix (simply write $\G_t \G_t^\herm = \sigma^2\I_N + \U_t (\D_t - \sigma^2\I_N) \U_t^\herm$).

 	Notice that $m_{\D_t}(z)=m_{\GG_t}(z-\sigma^2)$. Using
 	this fact, and the result in {\cite[Theorem 1.1]{SIL95}}, one can easily prove that
 	$m_{\frac{1}{M}\Y_t^{\herm}\Y_t}$ satisfies:
 	$$
 	\forall z\in \mathbb{C}\setminus \mathbb{R}^+\ ,\qquad m_{\frac{1}{M}\Y_t^{\herm}\Y_t}(z)-\underline{m}(z)\xrightarrow[M,N,n\to\infty]{\textrm{a.s.}}0\ ,
 	$$
 	where $\underline{m}(z)$ is the unique Stieltjes transform of
 	a probability distribution $\underline{F}$, solution of the following
 	functional equation:
 	\begin{equation}\label{eq:fixed-point}
 	\underline{m}(z)=\left( -z+\frac{N}{M}\int
 	\frac{\lambda+\sigma^2}{1+(\lambda+\sigma^2)
 	\underline{m}(z)}d\, F^{\GG_t}(\lambda)\right)^{-1}\ .
 	\end{equation}
 	Moreover, $\underline{m}(z)$ is analytical on $\CC^+=\left\{z\in\CC,
 	\Im(z)> 0\right\}$ where $\Im(z)$ stands for the imaginary part of
 	$z\in \CC$. Using \eqref{eq:fixed-point}, one can prove that
 	$m_{\GG_t}(z)$ satisfies:
 	\begin{equation}\label{eq:unobservable}
 	m_{\GG_t}\left(-\frac{1}{\underline{m}(z)}-\sigma^2\right)
 	=\underline{m}(z)(1-\frac{M}{N})-\frac{M}{N}\,z\, \underline{m}^2(z)\ .
 	\end{equation}
 	The link between the unobservable Stieltjes transform $
 	m_{\GG_t}$ and the deterministic equivalent
 	$\underline{m}(z)$ being established, it remains to express
 	$N^{-1}\log\det(\I_N+\sigma^{-2} \GG_t)$ in terms
 	of $ m_{\GG_t}$, which follows easily by differentiation:
 	$$
 	\frac{\partial}{\partial \sigma^2}\frac{1}{N}\log\det\left
 	(\I_N+\frac{\GG_t}{\sigma^2}\right)
 	=\frac{1}{N}\tr\left(\GG_t+\sigma^2\I_N\right)^{-1}-\frac{1}{\sigma^2}\ .
 	$$
 	Hence:
 	\begin{eqnarray}
 	\frac{1}{N}\log\det\left(\I_N+\frac{\GG_t}{\sigma^2}\right)
 	&=&\int_{\sigma^2}^{+\infty}\frac{1}{v}-\frac{1}{N}\tr\left(\GG_t+v\I_N\right)^{-1}dv\nonumber\ ,\\
 	&=&\int_{0}^{\frac{1}{\sigma^2}} \frac{1}{v}-\frac{1}{v^2}m_{\GG_t}\left(-\frac{1}{v}\right)dv \ .\label{eq:capacity_b}
 	\end{eqnarray}
 	We shall now perform a change of variables within the integral in
 	order to substitute $\underline{m}$ for $m_{\GG_t}$ with
 	the help of \eqref{eq:unobservable}. Since the support of $\underline{F}$ is on $\left[0,+\infty\right[$, the Stieltjes transform $\underline{m}$ is continuous and increasing on $\left]-\infty,0\right[$. It establishes then a bijection from $\left]-\infty,0\right[$ to $\left]\lim_{x\to -\infty}\underline{m}(x),\lim_{x\to 0}\underline{m}(x)\right[$. Obviously,  $\lim_{x\to -\infty}\underline{m}(x)=0$ whereas $\lim_{x\to 0^{-}}\underline{m}(x)=-\infty$ since $0$ is an eigenvalue of $\frac{1}{M}{\bf Y}_t{\bf Y}_t^{\herm}$ with multiplicity at least equal to $M-N$. 
	
 We have thus, 
 	$$
 	u\mapsto\left(\frac{1}{\underline{m}(u)}+\sigma^2\right)^{-1}
 	$$
 	establishes a bijection from $\RR_{-}^*$ to $(0,1/\sigma^2)$. Considering the change of variable
 	$\frac{1}{t}=\frac{1}{\underline{m}(u)}+\sigma^2$, \eqref{eq:capacity_b} writes:
 	\begin{eqnarray*}
 	\lefteqn{ \frac{1}{N}\log\det\left(\I_N+\frac{\GG_t}{\sigma^2}\right)}\\
 	&=&\int_{-\infty}^0\left[\frac{1}{\underline{m}(u)}+\sigma^2-\left(\frac{1}{\underline{m}(u)}+\sigma^2\right)^2
 	m_{\GG_t}\left(-\frac{1}{\underline{m}(u)}-\sigma^2\right)\right]\frac{\underline{m}'(u)}
 	{(1+\sigma^2\underline{m}(u))^2}\,du\\
 	&=&\int_{-\infty}^0\left[ \frac{\underline{m}'(u)}{\underline{m}(u)(1+\sigma^2\underline{m}(u))}-
 	\left(1-\frac{M}{N}\right)\frac{\underline{m}'(u)}{\underline{m}}+\frac{M}{N}\,u\,\underline{m}'(u)\right] du\\
 	&=&\int_{-\infty}^{0} \left[ \frac{M}{N}\frac{\underline{m}'(u)}{\underline{m}(u)}-\frac{\sigma^2\underline{m}'(u)}{1+\sigma^2\underline{m}(u)}+\frac{M}{N}u\underline{m}'(u)\right]du.
 	\end{eqnarray*}
 	We shall now compute this integral, denoted by $I$ in the
 	sequel. Write $I= \lim_{\substack{x\to{-\infty}\\ y\to 0}} I_{x,y}$ where
 	$$
 	I_{x,y}=\int_{x}^{y} \left[
 	\frac{M}{N}\frac{\underline{m}'(u)}{\underline{m}(u)}-\frac{\sigma^2\underline{m}'(u)}{1+\sigma^2\underline{m}(u)}+\frac{M}{N}u\underline{m}'(u)\right]\,du\ .
 	$$
 	Straightforward computations yield:
 	\begin{equation}
 	I_{x,y}=\log\left|\frac{(\underline{m}(y))^{\frac{M}{N}}}{1+\sigma^2\underline{m}(y)}\right|-\log\left|\frac{(\underline{m}(x))^{\frac{M}{N}}}{1+\sigma^2\underline{m}(x)}\right|+\frac{M}{N}y\underline{m}(y)-\frac{M}{N}x\underline{m}(x)-\int_{x}^{y}\frac{M}{N}\underline{m}(u)du\ .
 	\label{eq:Ialphabeta}
 	\end{equation}
 	As our objective is to compute the limit of $I_{x,y}$ as $x\to
 	-\infty$ and $y\to 0$, we need to obtain equivalents for
 	$\underline{m}$ at $0$ and $-\infty$. A direct application of the
 	dominated convergence theorem yields:
 	\begin{equation*}
 	\underline{m}(x)\underset{x\to-\infty}{\sim}-\frac{1}{x}\ .
 	\end{equation*}
 	Recall that $\underline{F}$ is the probability distribution associated
 	to $\underline{m}$. Then,
 	$\underline{F}(\{0\})=M^{-1}(M-N)$. Although this property is not easy
 	to write down properly, it is quite intuitive if one sees
 	$\underline{F}$ a.s. close to $F^{\Y_t^{\herm}\Y_t}$ (the empirical
 	distribution of the eigenvalues of $\Y_t^{\herm}\Y_t$) which clearly
 	satisfies $F^{\Y_t^{\herm}\Y_t}(\{0\})=M^{-1}(M-N)$ by Assumption
 	\ref{ass:asymptotics}: This assumption implies in fact that zero is an
 	eigenvalue of $\Y_t^{\herm}\Y_t$ of order $M-N$. Hence,
 	\begin{equation*}
 	\underline{m}(y)\underset{y\to 0}{\sim} -\frac{M-N}{My}\ .
 	\end{equation*}
 	Using these relations, we can derive equivalents for the first four terms in the right-hand side of \eqref{eq:Ialphabeta}. In particular, we obtain:
 	\begin{eqnarray}
 	\log\left|\frac{(\underline{m}(y))^{\frac{M}{N}}}{1+\sigma^2\underline{m}(y)}\right|&\underset{y\to 0}{\sim}&
 	\left(\frac{M}{N}-1\right)\log\left(\frac{M-N}{M}\right)-\log(\sigma^2)+\left(1-\frac{M}{N}\right)\log|y|
 	\label{eq:1}\ , \\
 	-\log\left|\frac{(\underline{m}(x))^{\frac{M}{N}}}{1+\sigma^2\underline{m}(x)}\right|&\underset{x\to -\infty}{\sim}&
 	\frac{M}{N}\log|x|
 	\label{eq:2}\ ,\\
 	\frac{M}{N}\,y\, \underline{m}(y)&\underset{y\to 0}{\sim}& -\left(\frac{M}{N}-1\right)
 	\label{eq:3}\ ,\\
 	-\frac{M}{N}\, x\, \underline{m}(x)&\underset{x\to -\infty}{\sim}& \frac{M}{N}
 	\label{eq:4}\ .
 	\end{eqnarray}
 	Let us now handle the last term in \eqref{eq:Ialphabeta}. Clearly, we have:
	$$
	F^{\frac{1}{M}\Y_t^{\herm}\Y_t}(dx)=\frac{N}{M}	F^{\frac{1}{M}\Y_t\Y_t^{\herm}}+\frac{(M-N)}{M}\delta_0(dx)
	$$
	which implies that 
	$$
	m_{\frac{1}{M}\Y_t^{\herm}\Y_t}(z)=\frac{M}{N}	m_{\frac{1}{M}\Y_t\Y_t^{\herm}}(z)+\frac{(M-N)}{N}\frac{1}{z}
	$$
	The above relations can be also transferred to the limit Stieltjes transforms $m$ and $\underline{m}$ and their associated probability distribution functions $F$ and $\underline{F}$. 
	Actually, we have:

 	$$
 	\underline{F}(dx) = \frac{(M-N)}M \delta_{0}(dx) + \frac{N}M F(dx)\ .
 	$$
	and also:
 	$$
 	m(z) = \frac MN \underline{m}(z) + \frac {(M-N)}N \, \frac 1z\ .
 	$$
 	Note in particular that $m_{\Y_t \Y_t^{\herm}} - m\rightarrow 0$, hence that $F$ is a deterministic approximation of $F^{\frac{1}{M}\Y_t \Y_t^{\herm}}$, the empirical distribution of the eigenvalues of $\frac{1}{M}{\Y_t \Y_t^{\herm}}$. Now,
 	\begin{eqnarray}
 	\int_{x}^{y}\frac{M}{N}\underline{m}(u)du&=&\int_{x}^{y} \int \frac{d\,F(t)}{t-u}du-\frac{M-N}{Nu}du\ ,\nonumber\\
 	&=&\int (-\log|t-y|+\log|t-x|)d\,F(t)+\frac{M-N}{N}\left(\log|x|-\log|y|\right)\label{eq:integral}\ .
 	\end{eqnarray}
 	Using the dominated convergence theorem, one can prove that the r.h.s. of \eqref{eq:integral} is equivalent to:
 	\begin{equation}
 	\int_{x}^{y}\frac{M}{N}\underline{m}(u)du\underset{{\substack{x\to{-\infty}\\ y\to 0}}}{\sim} -\int\log(t) dF(t)+\frac{M}{N}\log|x|-\frac{M-N}{N}\log|y|\ .
 	\label{eq:5}
 	\end{equation}
 	Plugging \eqref{eq:1}, \eqref{eq:2}, \eqref{eq:3}, \eqref{eq:4} and \eqref{eq:5} into \eqref{eq:Ialphabeta} yields:
 	$$
 	\lim_{\substack{x\to{-\infty} \\ y\to 0}} I_{x,y}=\frac{M-N}{N}\log\left(\frac{M-N}{M}\right)-\log\sigma^2+\int\log(t) dF(t)+1.
 	$$
 	Since the spectrum of $\frac{1}{M}\Y_t\Y_t^{\herm}$ is almost surely
 	eventually bounded away from zero and upper-bounded \cite{SIL98}, uniformly along
 	$N$, we have:
 	$$
 	\frac{1}{N}\sum_{i=1}^N \log(\lambda_i)-\int \log(t) dF(t)\xrightarrow[M,N,n\to+\infty]{\mathrm{a.s.}}{0}
 	$$
 	where $(\lambda_i, 1\le i\le N)$ are the eigenvalues of
 	$\frac{1}{M}\Y_t\Y_t^{\herm}$. A consistent estimator of
 	$\frac{1}{N}\log\det(\G_t\G_t^{\herm})$ is thus given by:
 	\begin{eqnarray*}
 	I_1&=&\frac{M-N}{N}\log\left(\frac{M-N}{M}\right)+1+\frac{1}{N}\sum_{i=1}^N\log(\lambda_i)\\
 	&=&\frac{M-N}{N}\log\left(\frac{M-N}{M}\right)+1+\frac{1}{N}\log\det\left(\frac{1}{M}\Y_t\Y_t^{\herm}\right)\ ,
 	\end{eqnarray*}
 	which concludes the proof.

 	\section{Proof of lemma \ref{lemma:deter}}
 	\label{app:lemma_deter}
 	Define for $\rho\geq 0$:
 	\begin{eqnarray*}
 	\Q_t(\rho,y)&=&\left(\rho \I_N+y\H_t\H_t^{\herm}+\frac{1}{M}\Y_t\Y_t^{\herm}\right)^{-1}\ ,\\
 	g_t(\rho,y)&=&\frac{1}{N}\log\det\left(\rho \I_N+y\H_t\H_t^{\herm}+\frac{1}{M}\Y_t\Y_t^{\herm}\right)\ .
 	\end{eqnarray*}

Recall that $\Y_t=\G_t \W_t$. Denote
 	by $\G_t=\U_t  \D_t^{\frac{1}{2}} \V_t^{\herm}$ the
 	singular value decomposition of $\G_t$, $\D_t$ being the diagonal matrix of
 	eigenvalues of $\G_t\G_t^{\herm}$; in particular,
 	$\D_t$'s entries are nonnegative and bounded away from zero. Let
 	$\widetilde{\W}_t=\V^{\herm}_t \W_t$. Since
 	the entries of $\W_t$ are i.i.d. and Gaussian, $\widetilde{\W}_t$ has
 	the same entry distribution as $\Z_t$. Hence
 	$g_t(\rho,y)$ becomes:
 	\begin{align*}
 	g_t(\rho,y)&=\frac{1}{N}\log\det\left(\rho\I_N+y\H_t\H_t^{\herm}+\frac{1}{M}\U_t\D_t^{\frac{1}{2}}\widetilde{\W}_t
 	\widetilde{\W}^{\herm}_t\D_t^{\frac{1}{2}}\U_t^{\herm}\right)\ ,\\
 	&=\frac{1}{N}\log\det\left(\rho\I_N+y\U_t^{\herm}\H_t\H_t^{\herm}\U_t
 	+\frac{1}{M}\D_t^{\frac{1}{2}}\widetilde{\W}_t \widetilde{\W}_t^{\herm}\D_t^{\frac{1}{2}}\right).
 	\end{align*}
 	Obviously, we have $-\frac{1}{N}\log\det(\Q_t(y))=g_t(0,y)$ and
 	$\frac{1}{M}\tr\Q_t(y)=\frac{1}{M}\tr\Q_t(0,y)$. Deterministic
 	equivalents for $g_t(\rho,y)$ and $\Q_t(\rho,y)$ have been derived in
 	\cite{HAC07} and are recalled in the lemma below.
 	\begin{lemma}[cf. \cite{HAC07}] Let $\rho>0$.
 	\begin{enumerate}
 	\item Let $y>0$. The following
 	functional equation:
 	$$
 	\kappa_t(\rho,y)=\frac{1}{M}\tr\left(\G_t \G_t^\herm\left(\rho\I_N+y\H_t\H_t^{\herm}+\frac{\G_t \G_t^\herm}{1+\kappa_t(\rho,y)}\right)^{-1}\right)
 	$$
 	admits a unique positive solution $\kappa_t(\rho,y)$.
 	\item Define
 	$$
 	\T_t(\rho,y)=\left(\rho \I_N+y\H_t\H_t^{\herm}+\frac{\G_t \G_t^\herm}{1+\kappa_t(\rho,y)}\right)^{-1}.
 	$$
 	Then, for any sequence of deterministic matrices ${\bf S}_N\in \CC^{N\times N}$
 	with uniformly bounded spectral norm:
 	$$
 	\frac{1}{M}\tr{\bf S}_N\Q_t(\rho,y)-\frac{1}{M}\tr{\bf S}_N
 	\T_t(\rho,y)\xrightarrow[M,N,n\rightarrow \infty]{\mathrm{a.s.}} 0\ .
 	$$
 	In particular, setting ${\bf S}_N=\G_t \G_t^\herm$, we get:
 	$$
 	\frac{1}{M}\tr \G_t \G_t^\herm \Q_t(\rho,y)-\kappa_t(\rho,y)\xrightarrow[M,N,n\rightarrow \infty]{\mathrm{a.s.}} 0\ .
 	$$
 	\item Let
 	$$
 	V_t(\rho,y)=\log\det\left(\rho \I_N+y\H_t\H_t^{\herm}+\frac{\G_t \G_t^\herm}{1+\kappa_t(\rho,y)}\right)+{M}\log(1+\kappa_t(\rho,y))-{M}\frac{\kappa_t(\rho,y)}{1+\kappa_t(\rho,y)}\ ,
 	$$
 	then
 	$$
 	g(\rho,y)-\frac{1}{N}V_t(\rho,y)\xrightarrow[M,N,n\rightarrow \infty]{\mathrm{a.s.}} 0\ .
 	$$
 	\end{enumerate}
 	\label{lemma:equivalent_deterministe}
 	\end{lemma}
 	
 	The general idea of the proof of Lemma \ref{lemma:deter} is to
 	transfer these deterministic equivalents to the case $\rho\searrow 0$; we will proceed by taking advantage
 	from the fact that all the diagonal elements of
 	$\D_t$ are positive and uniformly bounded away from zero.
 	
 	We first prove the existence and uniqueness of
 	$\kappa_t(y)$. Consider the function $f$ defined on $\left[0,\infty\right[$ by:
 	$$
 	f:x\mapsto x-\frac{1}{M}\tr\D_t\left(y\U_t^{\herm}\H_t\H_t^{\herm}\U_t+\frac{\D_t}{1+x}\right)^{-1}\ .
 	$$
 	An easy computation yields the derivative of $f$ with respect to $x$:
 	$$
 	f'(x)=1-\frac{1}{M}\tr\D_t\left(y\U_t^{\herm}\H_t\H_t^{\herm}\U_t+\frac{\D_t}{1+x}\right)^{-1}\frac{\D_t}{(1+x)^2}\left(y\U_t^{\herm}\H_t\H_t^{\herm}\U_t+\frac{\D_t}{1+x}\right)^{-1}
 	$$
 	which is obviously always positive. Function $f$ is thus always increasing and thus establishes a
 	bijection from $\left[0,\infty\right[$ to $\left[f(0),+\infty\right[$. Since $f(0)$ is negative, we conclude
 	that $f$ has a single zero. This proves the existence and uniqueness of $\kappa_t(y)$.
 	It remains to
 	extend the asymptotic convergence results to the case $\rho = 0$.
 	
 	In the sequel, we only prove item 2) for ${\bf S}_N=\G_t\G_t^{\herm}$ as it
 	captures the key arguments of the proof; the extension to general
 	sequences $({\bf S}_N)$ will then be straightforward. Write
 	$\frac{1}{M}\tr \G_t \G_t^\herm\Q_t(y)-\kappa_t(y)$ as:
 	\begin{align*}
 	\frac{1}{M}\tr \G_t \G_t^\herm\Q_t(y)-\kappa_t(y)
 	&=\frac{1}{M}\tr \G_t \G_t^\herm\Q_t(y)-\frac{1}{M}\tr \G_t \G_t^\herm \Q_t(\epsilon,y)\\
 	&\quad +\frac{1}{M}\tr \G_t \G_t^\herm \Q_t(\epsilon,y)-\kappa_t(\epsilon,y)\\
 	&\quad \quad +\kappa_t(\epsilon,y)-\kappa_t(y)\ ,
 	\end{align*}
 	where $\epsilon >0$. We now handle sequentially each of the
 	differences of the r.h.s. of the previous decomposition. We
 	first prove that there exists a fixed constant $K>0$ (which only
 	depends on $\limsup NM^{-1}$) such that for every $\epsilon >0$, there
 	exists $N_1$ (which depends on the realization and hence is random)
 	such that for every $N\geq N_1$, we have:
 	\begin{equation}
 	\left|\frac{1}{M}\tr \G_t \G_t^\herm \Q_t(y)-\frac{1}{M}\tr \G_t \G_t^\herm \Q_t(\epsilon,y)\right|\leq \frac{\epsilon}{K}\ .
 	\label{prop:QH}
 	\end{equation}
	To prove this, we rely on the resolvent identity ${\bf
          B}^{-1}-{\bf C}^{-1}=-{\bf B}^{-1}\left({\bf B}-{\bf
            C}\right){\bf C}^{-1}$ which holds for any square
        invertible matrices ${\bf B}$ and ${\bf C}$.  Then, we have:
 	\begin{align*}
 \left|	\frac{1}{M}\tr \G_t \G_t^\herm \Q_t(y)-\frac{1}{M}\tr \G_t \G_t^\herm \Q_t(\epsilon,y)\right|&=\left|\frac{\epsilon}{M}\tr \G_t \G_t^\herm \Q_t(0,y)\Q_t(\epsilon,y)\right|\ \\
 	&\leq \frac{\epsilon}{M}\tr \G_t \G_t^\herm  \left\|
 	\left(\frac{1}{M}\D_t^{\frac{1}{2}}\widetilde{\W}\widetilde{\W}^{\herm}\D_t^{\frac{1}{2}}\right)^{-1}
 	\right\|^2\ .
 	\end{align*}
 	Recall that $\widetilde{\W}_t$ is an $N\times M$ matrix and that by
 	Assumption \ref{ass:asymptotics}, $\limsup_{M,N} N M^{-1} <1$.
 	Therefore the spectrum of $\widetilde{\W}_t\widetilde{\W}_t^{\herm}$
 	is almost surely eventually bounded away from zero\footnote{Recall
 	that if $\lim NM^{-1}=c<1$, then the smallest eigenvalue
 	$\lambda_{\min}(\widetilde{\W}_t\widetilde{\W}_t^{\herm})$ converges
 	to $(1-\sqrt{c})^2>0$; it remains to argue on subsequences to conclude
 	in the case where $\limsup_{M,N} N M^{-1} <1$ .}. In
 	particular, there exists a constant $K$ such that eventually, we have
 	$\left\|\left(\frac{1}{M}\D_t^{\frac{1}{2}}\widetilde{\W}\widetilde{\W}^{\herm}\D_t^{\frac{1}{2}}\right)^{-1}
 	\right\|^2\leq K^{-1}$, hence:
 	$$
 	\exists N_1,\ \forall N\ge N_1,\quad \left|\frac{1}{M}\tr\G_t \G_t^\herm \Q_t(y)-\frac{1}{M}\tr\G_t \G_t^\herm \Q_t(\epsilon,y)\right|\leq
 	\frac{\epsilon}{K}\ .
 	$$
 	
 	The second step consists in proving that for some constant $\widetilde{K}$ (depending on $\limsup NM^{-1}$)
 	there exists $N_2$ (depending on the realization) such that for all $N\geq N_2$:
 	\begin{equation}
 	\left|\kappa_t(\epsilon,y)-\kappa_t(y)\right|\leq \widetilde{K}\epsilon
 	\label{eq:kappa_p}\ .
 	\end{equation}
 	The proof of \eqref{eq:kappa} relies on the following identity:
 	\begin{equation}
 	\kappa_t(y)-\kappa_t(\epsilon,y)=\epsilon \alpha_N+\beta_N\left(\kappa_t(y)-\kappa_t(\epsilon,y)\right)\ ,
 	\label{eq:kappa_initial}
 	\end{equation}
 	where
 	\begin{align*}
 	\alpha_N&=\frac{1}{M}\tr\G_t \G_t^\herm\T_t(\epsilon,y)\T_t(y)  \ ,\\
 	\beta_N&=\frac{1}{M}\tr\left(\frac{\G_t \G_t^\herm\T_t(\epsilon,y)\G_t \G_t^\herm\T_t(y)}{(1+\kappa_t(y))(1+\kappa_t(\epsilon,y))}\right)\ .
 	\end{align*}
 	It is clear that $\beta_N <\liminf \frac{N}{M}$. Thus, by Assumption \ref{ass:asymptotics}, $\beta_N <1$. Also, one can prove that there exists $\widetilde{K}>0$
 	such that $\limsup\alpha_N < \widetilde{K}$. In fact, $\alpha_N$ satisfies:
 	\begin{equation}
 	\alpha_N \leq \frac{N}{M} \left\|\G_t \G_t^\herm\right\|
 	\left\|\left( \G_t \G_t^\herm\right) ^{-1}\right\|^2(1+\kappa_t(y))(1+\kappa_t(\epsilon,y))\ .
 	\label{eq:kappa}
 	\end{equation}
 	One can prove that $\kappa_t(y)$ and $\kappa_t(\epsilon,y)$ are smaller
 	than $\frac{N}{M(1-N/M)}$. In fact, $\kappa_t(y)$ can be written as:
 	\begin{eqnarray*}
 	\kappa_t(y)&=&\frac{N(1+\kappa_t(y))}{M}-\frac{(1+\kappa_t(y))}{M}\tr\left(y\H_t\H_t^{\herm}\left(y\H_t\H_t^{\herm}+\frac{\G_t \G_t^\herm}{1+\kappa_t(y)}\right)^{-1}\right)\ ,\\
 	&=&\frac{N}{M(1-\frac{N}{M})}-\frac{(1+\kappa_t(y))}{M(1-\frac{N}{M})}\tr\left(y\H_t\H_t^{\herm}\left(y\H_t\H_t^{\herm}+\frac{\G_t \G_t^\herm}{1+\kappa_t(y)}\right)^{-1}\right)\ ,\\
 	&\leq& \frac{N}{M(1-\frac{N}{M})}\ .
 	\end{eqnarray*}
 	Similar arguments hold for $\kappa_t(\epsilon,y)$, thus proving that $\lim\sup\alpha_N\leq \widetilde{K}$.
 	From \eqref{eq:kappa_initial}, we conclude that there exists $N_3$ such that for all $N \geq N_3$,
 	$$|\kappa_t(\epsilon,y)-\kappa_t(y)|\leq \widetilde{K}\epsilon\ .
 	$$
 	We are now in position to prove the almost sure convergence of $\frac{1}{M}\tr\G_t \G_t^\herm\Q_t(y)-\kappa_t(y)$.
 	Consider the constants $K$ and $\widetilde{K}$ as defined previously and let $\epsilon >0$.
 	According to \eqref{prop:QH}, there exists $N_1$ such that:
 	$$
 	\forall N\ge N_1\ ,\quad \left|\frac{1}{M}\tr \G_t \G_t^\herm \Q_t(y)-\frac{1}{M}\tr
 	\G_t \G_t^\herm\Q_t(\epsilon,y)\right|\leq
 	\frac{\epsilon}{K}\ .
 	$$
 	Using the almost sure convergence result of
 	$\frac{1}{M}\tr\G_t \G_t^\herm \Q_t(\epsilon,y)$ stated in
 	Lemma \ref{lemma:equivalent_deterministe}, there exists
 	$N_2$ such that:
 	$$
 	\forall N\ge N_2\ ,\quad \left|\frac{1}{M}\tr\G_t \G_t^\herm \Q_t(\epsilon,y)-\kappa_t(\epsilon,y)\right|\leq \epsilon\ .
 	$$
 	Finally from \eqref{eq:kappa_p}, there exists $N_3$ such that for all $N\geq N_3$:
 	$$
 	|\kappa_t(\epsilon,y)-\kappa_t(y)|\leq \widetilde{K}\epsilon\ .
 	$$
 	Combining all these results, we have, for $N\geq \max(N_1,N_2,N_3)$:
 	$$
 	\left|\frac{1}{M}\tr\G_t \G_t^\herm\Q(y)-\kappa_t(y)\right|\leq \epsilon\left(\frac{1}{K}+1+\widetilde{K}\right)\ ,
 	$$
 	hence proving that:
 	$$
 	\frac{1}{M}\tr\G_t \G_t^\herm\Q_t(y)-\kappa_t(y)\xrightarrow[M,N,n\rightarrow\infty]{\textrm{a.s.}} 0\ ,
 	$$
 	which is the desired result.
 	
 	\section{Proof of Theorem \ref{th:gestimator}}
 	\label{app:g_estimator}
 	As previously mentionned, the proof of Theorem \ref{th:gestimator} relies on the existence of a consistent estimate for
 	$$
 	\II_{t,1}=\frac{1}{N}\log\det(\G_t\G_t^{\herm}+\H_t\H_t^{\herm})\ .
 	$$
 	Denote by $f(y)$ the parametrized quantity:
 	$$
 	f(y)=\frac{1}{N}\log\det(\Y_t\Y_t^{\herm}+y\H_t\H_t^{\herm})\ .
 	$$
 	Then by Lemma \ref{lemma:deter}-3), we obtain:
 	\begin{equation}\label{eq:equiv-f}
 	-f(y)+\frac{1}{N}\log\det\left(\frac{\G_t\G_t^{\herm}}{1+\kappa_t(y)}+y\H_t\H_t^{\herm}\right)
 	+\frac{M}{N}\log(1+\kappa_t(y))-\frac{M}{N}\frac{\kappa_t(y)}{1+\kappa_t(y)} \xrightarrow[M,N,n\rightarrow\infty]{\textrm{a.s.}} 0\ .
 	\end{equation}
 	Obviously, if $y$ is replaced by $y_{N,t}$, a solution of:
 	\begin{equation}\label{eq:y-def-bis}
 	y_{N,t}=\frac{1}{1+\kappa_t(y_{N,t})}\ ,
 	\end{equation}
 	then the term $C_{t,1}$ appears in \eqref{eq:equiv-f}. The existence and uniqueness of $y_{N,t}$ immediately follows from the fact that the function $g$ defined as:
 	$$
 	g:x\mapsto (1+x)\frac{1}{M}\tr(\G_t\G_t^{\herm})(\H_t\H_t^{\herm}+\G_t\G_t^{\herm})^{-1}
 	$$
 	is a contraction. Moreover, straightforward computations yield:
 	\begin{equation}\label{eq:y-def-ter}
 	y_{N,t}=1-\frac{1}{M}\tr\G_t\G_t^{\herm}(\H_t\H_t^{\herm}+\G_t\G_t^{\herm})^{-1}\ .
 	\end{equation}
 	Unfortunately, $y_{N,t}$ depends on the unobservable matrix $\G_t$. One needs therefore to provide a consistent estimate $\yest$ of $y_{N,t}$. In order to proceed, we shall study the asymptotics of $\kappa_t(y)$. By Lemma \ref{lemma:deter}-2), we have:
 	\begin{equation}
 	\frac{y}{M}\tr \H_t\H_t^{\herm}\Q_t(y)-\frac{y}{M}\tr \H_t\H_t^{\herm}\T_t(y)\xrightarrow[M,N,n\rightarrow\infty]{\textrm a.s.} 0\ .
 	\label{eq:deter1}
 	\end{equation}
 	On the other hand, we have:
 	\begin{align}
 	\frac{y}{M}\tr \H_t\H_t^{\herm}\T_t(y)&=\frac{1}{M}\tr y\, \H_t\H_t^{\herm} \left(y\H_t\H_t^{\herm}+\frac{\G_t\G_t^{\herm}}{1+\kappa_t(y)}\right)^{-1}\ ,\nonumber\\
 	&=\frac{N}{M}-\frac{1}{M(\kappa_t(y)+1)}\tr\left(\G_t\G_t^{\herm}\left(y\H_t\H_t^{\herm}+\frac{\G_t\G_t^{\herm}}{1+\kappa_t(y)}\right)^{-1}\right)\ ,\nonumber\\
 	&=\frac{N}{M}-\frac{\kappa_t(y)}{1+\kappa_t(y)}\ ,\nonumber\\
 	&=\frac{N}{M}-1+\frac{1}{1+\kappa_t(y)}\ .\label{eq:deter2}
 	\end{align}
 	Substituting \eqref{eq:deter2} into \eqref{eq:deter1}, we obtain:
 	\begin{equation}
 	\frac{1}{M}\tr y\H_t\H_t^{\herm}\Q_t(y)-\frac{N}{M}+1-\frac{1}{\kappa_t(y)+1}\xrightarrow[M,N,n\rightarrow\infty]{\textrm{a.s.}} 0\ .
 	\label{eq:deter3}
 	\end{equation}
 	Intuitively, a consistent estimate $\yest$ of $y_{N,t}$ should satisfy $\yest=M^{-1} \yest\, \tr \H_t\H_t^{\herm}\Q_t(\yest)-\frac{N}{M}+1$. This intuition is confirmed by the following lemma:
 	\begin{lemma}\label{lemma:approx-point-fixe}
 	There exists a unique positive solution $\yest$ to the equation:
 	$$
 	\frac{\yest}{M}\tr \H_t\H_t^{\herm}\Q_t(\yest)-\frac{N}{M}+1-\yest=0\ .
 	$$
 	Moreover, the following convergence holds true:
 	$$
 	\yest-y_{N,t}\xrightarrow[M,N,n\rightarrow\infty]{\textrm a.s.} 0\ ,
 	$$
 	where $y_{N,t}$ is defined by \eqref{eq:y-def-bis} (see also \eqref{eq:y-def-ter}).
 	\end{lemma}
 	\begin{proof}
The existence of $\yest$ follows from the fact that: $h:y\mapsto \frac{y}{M}\tr\H_t\H_t^{\herm}\Q_t(y)-\frac{N}{M}+1-y$ is a continuous  function on $\left[0,+\infty\right[$, satisfying $h(0)>0$ and $\lim_{y\to+\infty} h(y)=-\infty$. Assume that $h$ admits more than one zero. It is clear that the zeros of $h$ are isolated. Since $h(0) >0$, there exists then $y_1$ and $y_2$ such that $h(y_1)=h(y_2)=0$ and $h(y)<0$ for every $y\in\left[y_1,y_2\right]$. However, this could not happen since $h$ is concave, and as such $h(\frac{y_1+y_2}{2})\geq 1/2h(y_1)+1/2h(y_2)=0$. Function $h$ admits then a unique zero $\yest$.

 	Using \eqref{eq:deter3}, we get that:
 	$$
 	\frac{ y_{N,t}}{M}\tr\H_t\H_t^{\herm}\Q_t(y_{N,t})-\frac{N}{M}+1-y_{N,t}\xrightarrow[M,N,n\rightarrow\infty]{\textrm{a.s.}} 0\ .
 	$$
 	Beware that in \eqref{eq:deter3}, the convergence holds true for a fixed $y$ while $y_{N,t}$ depends upon $N$. A way to circumvent this issue is to merge $y_{N,t}$ into $\H_t$
 	and to consider the slightly different model based on $\widetilde{\H}_t= \sqrt{y_{N,t}} \H_t$.
 	
 	Therefore, the mere definition of $\yest$ and the previous convergence yield:
 	$$
	k(y_{N,t},\yest)\xrightarrow[M,N,n\rightarrow\infty]{\textrm{a.s.}} 0,
 	$$
where
$$
	k(y_{N,t},\yest)=\frac{\yest}{M}\tr(\H_t\H_t^{\herm}\Q_t(\yest))-\yest+y_{N,t}-\frac{y_{N,t}}{M}\tr(\H_t\H_t^{\herm}\Q_t(y_{N,t})).
$$
 Expanding  $k(y_{N,t},\yest)$, we get:
	\begin{align*}
	k(y_{N,t},\yest)&=\frac{\yest}{M}\tr(\H_t\H_t^{\herm}\Q_t(\yest))-\frac{\yest}{M}\tr(\H_t\H_t^{\herm}\Q_t(y_{N,t}))+\frac{\yest}{M}\tr(\H_t\H_t^{\herm}\Q_t(y_{N,t}))+(y_{N,t}-\hat{y}_{N,t})\\
	&-\frac{y_{N,t}}{M}\tr(\H_t\H_t^{\herm}\Q_t(y_{N,t}))\\
	&=\yest(y_{N,t}-\yest)\frac{1}{M}\tr(\H_t\H_t^{\herm}\Q_t(\yest)\H_t\H_t^{\herm}\Q_t(y_{N,t}))+(y_{N,t}-\yest)\\
&\quad +(\yest-y_{N,t})\frac{1}{M}\tr(\H_t\H_t^{\herm}\Q_t(y_{N,t}))\\
	&=\left(y_{N,t}-\yest\right)\left(1-\frac{1}{M}\tr\H_t\H_t^{\herm}\Q_t(y_{N,t})+\frac{\yest}{M}\tr(\H_t\H_t^{\herm}\Q_t(\yest)\H_t\H_t^{\herm}\Q_t(y_{N,t}))\right)
	\end{align*}
	To conclude that $y_{N,t}-\yest$ converges almost surely zero, one needs to estabslish that a deterministic asymptotic approximate of  $$\left(1-\frac{1}{M}\tr\H_t\H_t^{\herm}\Q_t(y_{N,t})+\frac{\yest}{M}\tr(\H_t\H_t^{\herm}\Q_t(\yest)\H_t\H_t^{\herm}\Q_t(y_{N,t}))\right)$$ could not be equal to zero. This is true, since from the definition of $y_{N,t}$, we can easily check that  $1-\frac{1}{M}\tr\H_t\H_t^{\herm}\Q_t(y_{N,t})$ can be approximated asymptotically by $1-\frac{1}{My_{N,t}}\tr\H_t\H_t^{\herm}\left(\H_t\H_t^{\herm}+\G_t\G_t^{\herm}\right)^{-1}$, where we recall that $y_{N,t}$ writes as:
	$$
	y_{N,t}=1-\frac{1}{M}\G_t\G_t^{\herm}\left(\H_t\H_t^{\herm}+\G_t\G_t^{\herm}\right)^{-1}=1-\frac{N}{M}+\frac{1}{M}\tr\H_t\H_t^{\herm}\left(\G_t\G_t^{\herm}+\H_t\H_t^{\herm}\right)^{-1}.
	$$
	The deterministic equivalent of $1-\frac{1}{M}\tr\H_t\H_t^{\herm}\Q_t(y_{N,t})$ is thus given by:
	$$
	\frac{1-\frac{N}{M}}{1-\frac{N}{M}+\frac{1}{M}\tr\H_t\H_t^{\herm}\left(\G_t \G_t^\herm+\H_t\H_t^{\herm}\right)^{-1}}
	$$
	which is obviously uniformly lower-bounded by $1$. 
 	\end{proof}
 	With the help of Lemma \ref{lemma:approx-point-fixe}, the following convergence can be easily verified:
 	\begin{eqnarray*}
 	\kappa(\yest)-\kappa(y_{N,t})&\xrightarrow[M,N,n\rightarrow\infty]{\textrm a.s.}&0\ .
 	\end{eqnarray*}
	Let $h:y\mapsto -\frac{1}{N}\log\det(\Q_t(y))$, where $h'(y)=\frac{1}{N}\tr(\H_t\H_t^{\herm}\Q_t(y))\leq \frac{\|\H_t\|^2}{\lambda_{\rm min}(\frac{1}{M}\Y_t\Y_t^{\herm})}$. As the minimum eigenvalue of $\frac{1}{M}(\Y_t\Y_t^{\herm})$ is almost surely bounded away zero, function $h$ is Lipschitz. Therefore, the following convergence hold true:
	$$
\frac{1}{N}\log\det(\Q_t(y_{N,t}))-\frac{1}{N}\log\det(\Q_t(\yest))\xrightarrow[M,N,n\rightarrow\infty]{\textrm{a.s.}} 0\ ,
$$
 	We then get:
	$$
 	-f(\yest)+\frac{1}{N}\log\det(\G_t\G_t^{\herm}+\H_t\H_t^{\herm})-\frac{M-N}{N}\log(\yest)-\frac{M}{N}(1-\yest)\xrightarrow[M,N,n\rightarrow\infty]{\textrm a.s.}0\ ,
 	$$
 	which in turn implies that:
 	$$
 	\II_{t,1}-\frac{1}{N}\log\det(\yest\H_t\H_t^{\herm}+\Y_t\Y_t^{\herm})-\frac{M-N}{N}\log(\yest)-\frac{M}{N}(1-\yest)\xrightarrow[M,N,n\rightarrow\infty]{\textrm a.s.}0\ .
 	$$
 	Using this estimate of $\II_{t,1}$ together with the estimate of $\II_{t,2}$ as provided in Lemma \ref{lemma:stieltjes} immediately yields a consistent estimate for $\II_t(\sigma^2)=\II_{t,1}-\II_{t,2}$, and the theorem is proved.

 	\section{Proof of theorem \ref{th:clt_main}}
 	\label{app:clt_main}
 	The proof of Theorem \ref{th:clt_main} relies on the tools
        used in \cite{HAC06}, adapted for dealing with Gaussian
        random variables. Recall that $\trad(y)$ is given by:
 	$$
 	\trad(y)=\frac{1}{NT}\sum_{t=1}^T \log\det\left(y\H_t\H_t^{\herm}+\frac{1}{M}\Y_t\Y_t^{\herm}\right)-\log\det\left(\frac{1}{M}\Y_t\Y_t^{\herm}\right),
 	$$
 	where $\Y_t=\G_t \W_t$. Similarly, as in Appendix
        \ref{app:lemma_deter} and Appendix \ref{app:stieltjes}, we can
        prove that $\Y_t=\U_t\D_t^{\frac{1}{2}}\widetilde{\W}_t$ where
        $\widetilde{\W}_t$ is a $N\times M$ standard Gaussian matrix,
        and $\D_t$ is the $N\times N$ diagonal matrix containing the
        eigenvalues of $\G_t\G_t^{\herm}$. Then,
        $\trad(y)$ becomes:
 	\begin{align*}
 	\trad(y)&=\frac{1}{NT}\sum_{t=1}^T \log\det(y\H_t\H_t^{\herm}+\frac{1}{M}\U_t\D_t^{\frac{1}{2}}\widetilde{\W}_t\widetilde{\W}_t^{\herm}\D_t^{\frac{1}{2}}\U_t^{\herm})-\log\det(\frac{1}{M}\D_t^{\frac{1}{2}}\widetilde{\W}_t\widetilde{\W}_t^{\herm}\D_t^{\frac{1}{2}}),\\
 	&=\frac{1}{NT}\sum_{t=1}^T \log\det(y\D_t^{-\frac{1}{2}}\U_t^{\herm}\H_t\H_t^{\herm}\U_t\D_t^{-\frac{1}{2}}+\frac{1}{M}\widetilde{\W}_t\widetilde{\W}_t^{\herm})-\log\det(\frac{1}{M}\widetilde{\W}_t\widetilde{\W}_t^{\herm}),\\
 	&=\frac{1}{NT}\sum_{t=1}^T\log\det\left(y\D_t^{-\frac{1}{2}}\U_t^{\herm}\H_t\H_t^{\herm}\U_t\D_t^{-\frac{1}{2}}\left(\frac{1}{M}\widetilde{\W}_t\widetilde{\W}_t^{\herm}\right)^{-1}+\I_N\right).
 	\end{align*}
 	Denote by $\D_t^{-\frac{1}{2}}\U_t^{\herm}\H_t\H_t^{\herm}\U_t\D_t^{-\frac{1}{2}}=\widetilde{\U}_t\boldsymbol{\Lambda}_t\widetilde{\U}_t^{\herm}$ the eigenvalue decomposition of $\D_t^{-\frac{1}{2}}\U_t^{\herm}\H_t\H_t^{\herm}\U_t\D_t^{-\frac{1}{2}}$. Since $p_t$ is the rank of $\H_t\H_t^{\herm}$, matrix $\boldsymbol{\Lambda}_t$ has exactly $p_t$ non zero entries which we denote by $(\lambda_{i,t},1\leq i\leq p_t)$. We get that $\trad$ can be written as:
 	$$
 	\trad(y)=\frac{1}{NT}\sum_{t=1}^T\log\det\left(y\boldsymbol{\Lambda}_t\left(\frac{1}{M}\widetilde{\W}_t\widetilde{\W}_t^{\herm}\right)^{-1}+\I_N\right).
 	$$
 	Let $\boldsymbol{\Lambda}_{p_t,t}={\rm diag}\left(\lambda_{1,t},\ldots,\lambda_{p_t,t}\right)$. Obviously, only the diagonal elements of $\boldsymbol{\Lambda}_{p_t,t}$ contribute in the expression of $\trad(y)$.	Then, using \cite[Theorem 3.2.11]{MUI82}, we can prove that $\trad(y)$ can be written as:
 	$$
 	\trad(y)=\frac{1}{NT}\sum_{t=1}^T\log\det\left(y\boldsymbol{\Lambda}_{p_t,t}\left(\frac{1}{M}\widetilde{\W}_{p_t,t}\widetilde{\W}_{p_t,t}^{\herm}\right)^{-1}+\I_{p_t}\right),
 	$$
 	where $\widetilde{\W}_{p_t,t}$ is a $p_t\times M-N+p_t$ standard Gaussian matrix.
 	Let $\M=\frac{(M-N+p_t)}{My}\boldsymbol{\Lambda}_{p_t,t}^{-1}$, we finally get:
 	\begin{align*}
 	\trad(y)&=\frac{1}{NT}\sum_{t=1}^T\log\det\left(\frac{1}{M-N+p_t}\M^{\frac{1}{2}}\widetilde{\W}_{p_t,t}\widetilde{\W}_{p_t,t}^{\herm}\M^{\frac{1}{2}}+\I_{p_t}\right)-\log\det\left(\M\right)-\log\det\left(\frac{1}{M-N+p_t}\widetilde{\W}_{p_t,t}\widetilde{\W}_{p_t,t}^{\herm}\right)\\
 	&\triangleq\sum_{t=1}^T \hat{\II}_{{\rm ES,t}}(y).
 	\end{align*}
 	
 	Let $s=M-N+p_t$. By Assumptions \ref{ass:asymptotics} and \ref{ass:channel-matrix-Hb}-1), we have:
 	$$
 	0<\liminf_{M,N,n_0\to\infty}\frac{s}{p_t}\leq \limsup_{M,N,n_0\to\infty} \frac{s}{p_t} <+\infty.
 	$$
 	Moreover, Assumption \ref{ass:channel-matrix-H} and \ref{ass:channel-matrix-Hb}-2) implies that  matrix $\M$ satisfies:
 	$$
 	\sup_{N,M,n} \|\M\| <\infty \hspace{0.1cm} \textnormal{and} \hspace{0.1cm} \inf_{N,M,n} \frac{1}{s}\tr\M >0.
 	$$
 	We retrieve then the same model as in \cite{HAC06}, with the
        slight difference that $\hat{\II}_{{\rm ES,t}}(y)$ has an
        extra random term
        $\log\det\left(\frac{1}{M}\widetilde{\W}_{p_t,t}\widetilde{\W}_{p_t,t}^{\herm}\right)$. As
        we will see next, this has no impact on the applicability of
        the method and one can get the desired result by following the
        same lines of \cite{HAC06}. For ease of notation, we will drop
        next the subscripts $p_t$ and $t$ from all matrices. In
        particular, we consider to prove a CLT for the functional
        $\log\det(\frac{\rho}{s}\M^{\frac{1}{2}}\widetilde{\W}\widetilde{\W}^{\herm}\M^{\frac{1}{2}}+\I)-\log\det(\frac{1}{s}\M^{\frac{1}{2}}\widetilde{\W}\widetilde{\W}^{\herm}\M^{\frac{1}{2}})$
        where $\rho>0$, $\widetilde{\bf W}$ is an $p_t\times s$
        standard Gaussian matrix and $\M$ is an $p_t\times p_t$
        deterministic matrix.
	
 	The expression of the variance for this CLT will depend on some deterministic quantities which we recall hereafter.
 	\subsection{Notations}
 	Let $\Z=\M^{\frac{1}{2}}\widetilde{\W}$ and define the resolvent matrix ${\bf S}(z)$ by:
 	$$
 	{\bf S}(z)=\left(\frac{z}{s}\M^{\frac{1}{2}}\widetilde{\W}\widetilde{\W}^{\herm}\M^{\frac{1}{2}}+\I\right)^{-1}=\left(\frac{z}{s}\Z\Z^{\herm}+\I\right)^{-1},
 	$$
 	Let also $I_s(z)$ be given by:
 	$$
 	I_s(z)=\log\det\left(\frac{z}{s}\M^{\frac{1}{2}}\widetilde{\W}\widetilde{\W}^{\herm}\M^{\frac{1}{2}}+\I\right)=-\log\det{\bf S}(z).
 	$$
 	We introduce the following intermediate quantities:
 	$$
 	\beta(z)=\frac{1}{s}\tr\M{\bf S}, \hspace{0.2cm} \alpha(z)=\frac{1}{s}\tr{\M}\EE{\bf S}, \hspace{0.2cm} \textnormal{and} \hspace{0.2cm} \stackrel{o}{\beta}=\beta-\alpha.
 	$$
 	Matrix $\widetilde{\R}(z)$ is an $s\times s$ diagonal matrix defined by:
 	$$
 	\widetilde{\R}(z)=\tilde{r}\I_s,
 	$$
 	where $\tilde{r}=\frac{1}{1+z\alpha(z)}$. We also define $\R(z)$ the $p_t\times p_t$ diagonal matrix given by:
 	$$
 	\R(z)=\left(\I+z\tilde{r}\M\right)^{-1}={\textrm{ diag}}(r_i,1\leq i\leq p_t),
 	$$
	where $r_i=\frac{1}{1+z\tilde{r}m_i}$.
 	We also define $\delta(z)$ as the unique positive solution of the following equation:
 	$$
 	\delta(z)=\frac{1}{s}\tr\M\left(\I+\frac{z}{1+z\delta(z)}\M\right)^{-1},
 	$$
 	where the existence and uniqueness of $\delta(z)$ have already been proven in \cite{HAC06}.
 	Let $\boldsymbol{\Xi}$ and $\widetilde{\boldsymbol{\Xi}}$ be the $p_t\times p_t$ and $s\times s$ diagonal matrices defined by:
 	$$
 	\boldsymbol{\Xi}=\left(\I+\frac{z}{1+z\delta(z)}\M \right)^{-1} \hspace{0.1cm}\textnormal{and}\hspace{0.2cm} \widetilde{\boldsymbol{\Xi}}=\frac{1}{1+z\delta(z)}\I_s
 	$$
 	Define also $\gamma$, $\tilde{\delta}(z)$ and $\tilde{\gamma}$ as $\gamma=\frac{1}{s}\tr\M^2\boldsymbol{\Xi}^2$, $\tilde{\delta}(z)=\frac{1}{1+z\delta(z)}$ and $\widetilde{\gamma}=\frac{1}{(1+z\delta(z))^2}$.
 	\subsection{Mathematical tools}\label{app:gaussian-calculus}
 	We recall here the mathematical tools that will be used to establish theorem \ref{th:clt_main}. All these results can be found in \cite{HAC06}.
 	\begin{enumerate}
 	\item Differentiation formulas:
 	\begin{align*}
 	\frac{\partial S_{p,q}}{\partial {Z}_{i,j}}&=-\frac{z}{s}\left[\Z^{\herm}{\bf S}\right]_{j,q} S_{p,i},\\
 	\frac{\partial S_{p,q}}{\partial Z_{i,j}^*}&=-\frac{z}{s}\left[{\bf S}\Z\right]_{p,j}S_{i,q},\\
 	\frac{\partial I_s(z)}{\partial Z_{i,j}^*}&=\frac{z}{s}\left[{\bf S}{\bf Z}\right]_{i,j},\\
 	\frac{\partial{\log\det(\frac{1}{s}{\bf Z}{\bf Z}^{\herm})}}{\partial Z_{i,j}^*}&=\left[\left({\bf Z}\Z^{\herm}\right)^{-1}{\bf Z}\right]_{i,j}.
 	\end{align*}
 	\item Integration by parts formula for Gaussian functionals: Denote by $\Phi$ be a $\mathcal{C}^1$ complex function polynomially bounded with its derivatives, then
 	$$
 	\EE\left[Z_{i,j}\Phi(\Z)\right]=m_i\EE\left[\frac{\partial \Phi(\Z)}{\partial Z_{i,j}^*}\right].
 	$$
 	where $m_i$ is the $i$-th diagonal element of $\M$.
 	\item Poincar\'{e}-Nash inequality: The variance of $\Phi(\Z)$ can be upper-bounded as:
 	$$
 	{\rm var}(\Phi(\Z))\leq \sum_{i=1}^{p_t}\sum_{j=1}^s m_i\EE\left[\left|\frac{\partial \Phi(\Z)}{\partial Z_{i,j}}\right|^2+\left|\frac{\partial \Phi(\Z)}{\partial Z_{i,j}^*}\right|^2\right].
 	$$
 	\item Deterministic approximations of some functionals:
 	\begin{proposition}
          Let $\A$ and $\B$ be two sequences of respectively
          $p_t\times p_t$ and $s\times s$ diagonal deterministic
          matrices with uniformly bounded spectral norm. Let
          Assumptions \ref{ass:channel}-\ref{ass:channel-matrix-H}
          hold true. Then, the following holds true:
 	$$
 	\frac{1}{s}\tr\A\R=\frac{1}{s}\tr\A\boldsymbol{\Xi}+\mathcal{O}\left(s^{-2}\right), \quad\tilde{r}=\tilde{\delta}+\mathcal{O}\left(s^{-2}\right) \quad \textnormal{and} \quad \EE\frac{1}{s}\tr\A\H=\frac{1}{s}\tr\A\boldsymbol{\Xi}+\mathcal{O}\left(s^{-2}\right).
 	$$
 	\label{prop:approximation}
 	\end{proposition}
 	\begin{proposition}
 	Let $\A$, $\B$ and ${\bf C}$ be three sequences of $p_t\times p_t$, $s\times s$ and $p_t\times p_t$ diagonal deterministic matrices whose spectral norm are uniformly bounded in $p_t$. Consider the following:
 	$$
 	\Phi(\Z)=\frac{1}{s}\tr\left(\A{\bf S}\frac{\Z\B\Z^{\herm}}{s}\right), \hspace{0.2cm}\Psi(\Z)=\frac{1}{s}\tr\left(\A{\bf S}\M{\bf S}\frac{\Z\B\Z^{\herm}}{s}\right),
 	$$
 	and assume that \ref{ass:channel}-\ref{ass:channel-matrix-H} hold true.
 	Then,
 	\begin{enumerate}
 	\item The following estimations hold true:
 	${\rm var}(\Phi(\Z)), {\rm var}(\Psi(\Z)), {\rm var}(\beta)$ are $\mathcal{O}\left(s^{-2}\right)$.
 	\item The following approximations hold true:
 	\begin{align}
 	\EE\left[\Phi(\Z)\right]&=\tilde{\delta}\frac{1}{s}\tr\A\M\boldsymbol{\Xi}+\mathcal{O}\left(s^{-2}\right),\\
 	\EE\left[\Psi(\Z)\right]&=\frac{1}{1-z^2\gamma\tilde{\gamma}}\left(\tilde{\delta}\frac{1}{s}\tr\B\frac{1}{s}\tr(\A\M^2\boldsymbol{\Xi}^2)-{z\gamma\tilde{\gamma}}\frac{1}{s}\tr\B\frac{1}{s}\tr\A\M\boldsymbol{\Xi}\right)+\mathcal{O}\left(s^{-2}\right),\\
 	\EE\left[\frac{1}{s}\tr\M{\bf S}\M{\bf S}\right]&=\frac{\gamma}{1-z^2\gamma\tilde{\gamma}}+\mathcal{O}\left(s^{-2}\right).
 	\end{align}
 	\end{enumerate}
 	\label{prop:estimation}
 	\end{proposition}
 	\end{enumerate}
 	\subsection{Central limit theorem}
 	All the notations being defined, we are now in position to
        show the CLT. We recall that our objective is to study the
        fluctuations of $\trad(y)=\sum_{t=1}^T
        \hat{\II}_{{\rm ES},t}(y)$. Since $\left(\hat{II}_{{\rm
              ES},t}(y),t=1,\cdots,T\right)$ are independent, it
        suffices to consider the CLT for $\hat{\II}_{{\rm ES},t}(y)$,
        for $t\in\left\{1,\cdots,T\right\}$. We consider thus the
        random quantity
        $I_s(z)-\log\det\left(\frac{1}{s}\Z\Z^{\herm}\right)$.
 	Before getting into the proof details, we shall first recall the CLT of $g(\Z)=-\log\det(\frac{1}{s}\Z\Z^{\herm})$ whose proof can be found in \cite{BAI04}. Indeed, it is shown that:
 	$$
 	\frac{-1}{\log(1-\frac{p_t}{s})}\left(-\log\det\left(\frac{1}{s}\Z\Z^{\herm}\right)-b_s\right)\xrightarrow[N,M,n\to\infty]{\mathcal{D}} \mathcal{N}(0,1).
 	$$
 	where $b_s=-p_t\left[\left(1-\frac{s}{p_t}\right)\log\left(1-\frac{p_t}{s}\right)-1\right]$.
 	Like in \cite{HAC06}, define $\Psi_s(u,z)=\EE\left[e^{\jmath u(I_s(z)-V_s(z)+g(\Z)-b_s)}\right]$, where $V_s(z)$ is the deterministic equivalent defined by:
 	$$
 	V_s(z)=s\log\left(1+z\delta(z)\right)+\log\det\left(\I+\frac{z}{1+z\delta(z)}\M\right)-sz\delta(z)\tilde{\delta}(z),
 	$$
 	and verifying:
 	$$
 	\frac{1}{s}\left(I_s(z)-V_s(z)\right)\xrightarrow[M,N,n\to\infty]{{\rm a.s}}{0}.
 	$$
 	The principle of the proof is to establish a differential equation verified by $\Psi_s(u,z)$. Writing the derivative of $\Psi_s(u,z)$ with respect to $z$, we get:
 	\begin{equation}
 	\frac{\partial \Psi_s}{\partial z}=\EE\left[\jmath u \frac{\partial I_s(z)}{\partial z}e^{\jmath u I_s(z)+\jmath u g(\Z)}\right]e^{-\jmath u V_s(z)-\jmath u b_s}-\jmath u \frac{d V_s(z)}{d z}\Psi_s(u,z).
 	\end{equation}
	Since $\frac{d V_s(z)}{d z}=s\delta\tilde{\delta}$ \cite{HAC06}, we have:
	\begin{equation}
 	\frac{\partial \Psi_s}{\partial z}=\EE\left[\jmath u \frac{\partial I_s(z)}{\partial z}e^{\jmath u I_s(z)+\jmath u g(\Z)}\right]e^{-\jmath u V_s(z)-\jmath u b_s}-\jmath u s\delta\tilde{\delta}\Psi_s(u,z).
 	\label{eq:derivative_psi_b}
 	\end{equation}

 	On the other hand, we have:
 	\begin{align*}
 	\EE\left[\frac{\partial I_s(z)}{\partial z} e^{\jmath u I_s(z)+\jmath u g(\Z)}\right]&=\EE \left[\tr\left(\frac{{\bf S}\Z\Z^{\herm}}{s}\right)e^{\jmath u I_s(z)+\jmath u g(\Z)}\right]\\
 	&=\frac{1}{s}\sum_{p,i=1}^{p_t} \sum_{j=1}^s \EE\left[Z_{i,j}S_{p,i}Z_{p,j}^* e^{\jmath u I(z)+\jmath u g(\Z)}\right].
 	\end{align*}
 	Applying the integration by part formula, we get:
 	\begin{align*}
 	\EE\left[Z_{i,j}S_{p,i}Z_{p,j}^*e^{\jmath u I_s(z)+\jmath u g(\Z)}\right]&=\EE\left[m_i \frac{\partial}{\partial Z_{i,j}^*}\left[S_{p,i}Z_{p,j}^*e^{\jmath u I(z)+\jmath u g(\Z)}\right]\right]\\
 	&= \EE\left[m_i S_{p,i}\delta(p-i)e^{\jmath u I(z)+\jmath u g(\Z)}\right]\\
 	&-\frac{z}{s}\EE\left[\left[{\bf S}\Z\right]_{p,j}m_iS_{i,i}Z_{p,j}^*e^{\jmath u I_s(z)+\jmath u g(\Z)}\right]\\
 	&+\frac{\jmath uz}{s} \EE\left[m_iS_{p,i}Z_{p,j}^* \left[{\bf S}\Z\right]_{i,j}e^{\jmath u I_s(z)+\jmath u g(\Z)}\right]\\
 	&+\EE\left[\jmath u m_i S_{p,i}Z_{p,j}^*\frac{\partial g(\Z)}{\partial Z_{i,j}^*}e^{\jmath u I_s(z)+\jmath u g(\Z)}\right].
 	\end{align*}
 	After summing over index $i$, we obtain:
 	\begin{align}
 	\EE\left[\left[{\bf S}\Z\right]_{p,j}Z_{p,j}^* e^{\jmath u I_s(z)+\jmath u g(\Z)}\right]&=\EE\left[m_p S_{p,p}e^{\jmath u I_s(z)+\jmath u g(\Z)}\right]\nonumber\\
 	&-\frac{z}{s}\EE\left[\tr({\M{\bf S}})\left[{\bf S}\Z\right]_{p,j}Z_{p,j}^* e^{\jmath u I_s(z)+\jmath u g(\Z)}\right]\nonumber\\
 	&+\frac{\jmath z u }{s}\EE\left[\left[{\bf S}\M{\bf S}\Z\right]_{p,j}Z_{p,j}^*e^{\jmath u I_s(z)+\jmath u g(\Z)}\right]\nonumber\\
 	&-\jmath u \EE\left[\left[{\bf S}\M\left(\Z\Z^{\herm}\right)^{-1}\Z\right]_{p,j}Z_{p,j}^* e^{\jmath u I_s(z)+\jmath u g(\Z)}\right]. \label{eq:SZ}
 	\end{align}
 	Recall the relation $\beta=\frac{1}{s}\tr\M{\bf S}$ and $\stackrel{o}{\beta}=\beta-\alpha$ where $\alpha=\frac{1}{s}\tr\M\EE{\bf S}$. Plugging the relation $\beta=\alpha+\stackrel{o}{\beta}$ into \eqref{eq:SZ}, we get:
 	\begin{align}
 	\EE\left[\left[{\bf S}\Z\right]_{p,j}Z_{p,j}^* e^{\jmath u I_s(z)+\jmath u g(\Z)}\right]&=\EE\left[m_p S_{p,p}e^{\jmath u I_s(z)+\jmath u g(\Z)}\right]-z\EE\left[\stackrel{o}{\beta}\left[{\bf S}\Z\right]_{p,j}Z_{p,j}^* e^{\jmath u I_s(z)+\jmath u g(\Z)}\right]\nonumber\\
 	&-z\alpha\EE\left[\left[{\bf S}\Z\right]_{p,j}Z_{p,j}^* e^{\jmath u I_s(z)+\jmath u g(\Z)}\right]+\frac{\jmath z u }{s}\EE\left[\left[{\bf S}\M{\bf S}\Z\right]_{p,j}Z_{p,j}^*e^{\jmath u I_s(z)+\jmath u g(\Z)}\right]\nonumber\\
 	&-\jmath u \EE\left[\left[{\bf S}\M\left(\Z\Z^{\herm}\right)^{-1}\Z\right]_{p,j}Z_{p,j}^* e^{\jmath u I_s(z)+\jmath u g(\Z)}\right]. \label{eq:SZ1}
 	\end{align}
 	Hence, solving this equation with respect to $\EE\left[\left[{\bf S}\Z\right]_{p,j}Z_{p,j}^* e^{\jmath u I_s(z)+\jmath u g(\Z)}\right]$ and using the fact that $\tilde{r}=\frac{1}{1+z\alpha}$, we get:
 	\begin{align}
 	\EE\left[\left[{\bf S}\Z\right]_{p,j}Z_{p,j}^* e^{\jmath u I_s(z)+\jmath u g(\Z)}\right]&=\EE\left[m_p\tilde{r}S_{p,p}e^{\jmath u I_s(z)+\jmath u g(\Z)}\right]
 	-z\EE\left[\stackrel{o}{\beta}\tilde{r}\left[{\bf S}\Z\right]_{p,j}Z_{p,j}^* e^{\jmath u I_s(z)+\jmath u g(\Z)}\right]\nonumber\\
 	&+\frac{z}{s}\EE\left[\jmath u \tilde{r}\left[{\bf S}\M{\bf S}\Z\right]_{p,j}Z_{p,j}^*e^{\jmath u I_s(z)+\jmath u g(\Z)}\right]\nonumber\\
 	&-\jmath u \EE\left[\tilde{r}\left[{\bf S}\M\left(\Z\Z^{\herm}\right)^{-1}\Z\right]_{p,j}Z_{p,j}^* e^{\jmath u I_s(z)+\jmath ug(\Z)}\right].\label{eq:useful_b}
 	\end{align}
 	Using the relation $S_{p,p}=1-\frac{z}{s}\left[{\bf S}\Z\Z^{\herm}\right]_{p,p}$, we get after summing with respect to $j$,
 	\begin{align*}
 	\EE\left[\left[\frac{{\bf S}\Z\Z^{\herm}}{s}\right]_{p,p}e^{\jmath u I_s(z)+\jmath u g(\Z)}\right]&=\EE\left[m_p \tilde{r}e^{\jmath u I_s(z)+\jmath u g(\Z)}\right]-zm_p\tilde{r}\left[\left[\frac{{\bf S}\Z\Z^{\herm}}{s}\right]_{p,p}e^{\jmath u I_s(z)+\jmath u g(\Z)}\right]\\
 	&-z\EE\left[\stackrel{o}{\beta}\tilde{r}\left[\frac{{\bf S}\Z\Z^{\herm}}{s}\right]_{p,p}e^{\jmath u I_s(z)+\jmath u g(\Z)}\right]+\frac{\jmath u z}{s}\EE\left[\tilde{r}\left[{\bf S}\M{\bf S}\frac{\Z\Z^{\herm}}{s}\right]_{p,p}e^{\jmath u I_s(z)+\jmath u g(\Z)}\right]\\
 	&-\jmath u \EE\left[\tilde{r}\left[\frac{{\bf S}\M}{s}\right]_{p,p}e^{\jmath u I_s(z)+\jmath u g(\Z)}\right].\label{eq:useful}
 	\end{align*}
 	Using the relation $r_p=\frac{1}{1+z\tilde{r}m_p}$, we have:
 	\begin{align*}
 	\EE\left[\left[\frac{{\bf S}\Z\Z^{\herm}}{s}\right]_{p,p}e^{\jmath u I_s(z)+\jmath u g(\Z)}\right]&=\EE\left[m_p r_p \tilde{r}e^{\jmath u I_s(z)+\jmath u g(\Z)}\right]-z\EE\left[\stackrel{o}{\beta}\tilde{r}r_p\left[\frac{{\bf S}\Z\Z^{\herm}}{s}\right]_{p,p}e^{\jmath u I_s(z)+\jmath u g(\Z)}\right]\\
 	&+\frac{\jmath u z}{s}\EE\left[\tilde{r}r_p\left[{\bf S}\M{\bf S}\frac{\Z\Z^{\herm}}{s}\right]_{p,p}e^{\jmath u I_s(z)+\jmath u g(\Z)}\right]-\jmath u \EE\left[\tilde{r}r_p\left[\frac{{\bf S}\M}{s}\right]_{p,p}e^{\jmath u I_s(z)+\jmath u g(\Z)}\right].
 	\end{align*}
 	Summing over $p$, we finally obtain:
 	\begin{align*}
 	\EE\left[\tr\left(\frac{{\bf S}\Z\Z^{\herm}}{s}\right)e^{\jmath u I_s(z)+\jmath u g(\Z)}\right]&=\tilde{r}\tr(\M\R)\EE\left[e^{\jmath u I_s(z)+\jmath u g(\Z)}\right]-z\EE\left[\stackrel{o}{\beta}\tilde{r}\tr\left(\R{\bf S}\frac{\Z\Z^{\herm}}{s}\right)e^{\jmath u I_s(z)+\jmath u g(\Z)}\right]\\
 	&+\frac{z}{s}\jmath u \EE\left[\tilde{r}\tr\left(\R{\bf S}\M{\bf S}\frac{\Z\Z^{\herm}}{s}\right)e^{\jmath u I_s(z)+\jmath u g(\Z)}\right]\\
 	&-\jmath u \tilde{r}\EE\left[\tr\left(\frac{\R{\bf S}\M}{s}\right)e^{\jmath u I_s(z)+\jmath u g(\Z)}\right]\\
 	&=\chi_1+\chi_2+\chi_3+\chi_4.
 	\end{align*}
 	It remains thus to deal with the terms $\left(\chi_i, 1\leq i \leq 4\right)$.
 	Using proposition \ref{prop:approximation}, we have:
 	\begin{equation}
 	\chi_1=\tilde{r}\tr\M\R\EE\left[e^{\jmath u I_s(z)+\jmath u g(\Z)}\right]=s\delta\tilde{\delta}\EE\left[e^{\jmath u I_s(z)+\jmath u g(\Z)}\right]+\mathcal{O}\left(s^{-1}\right).
 	\label{eq:chi_1}
 	\end{equation}
 	To deal with $\chi_3$, we apply the results of proposition \ref{prop:estimation}-b, with $\A=\R$ and $\B=\I$. In this case, $\chi_3$ writes as : $\chi_3=z\jmath u \tilde{r}\EE\Psi(\Z)e^{\jmath u I_s(z)+\jmath u g(\Z)}$. Using Cauchy-Schwartz inequality, we get:
 	$$
 	\left|\EE\left(\Psi(\Z) e^{\jmath u I_s(z)+\jmath u g(\Z)}\right)-\EE e^{\jmath u I_s(z)+\jmath u g(\Z)}\EE \left(\Psi(\Z)\right)\right|\leq \sqrt{\EE\left[\left|\stackrel{o}{\Psi}(\Z)\right|^2\right]},
 	$$
 	where $\stackrel{o}{\Psi}(\Z)=\Psi(\Z)-\EE\left(\Psi(\Z)\right)$.
 	Therefore,
 	\begin{equation}
 	\chi_3=\frac{z\jmath u \tilde{\delta}}{1-z^2\gamma\tilde{\gamma}}\left[\tilde{\delta}\frac{1}{s}\tr(\M^2\boldsymbol{\Xi}^3)-\frac{z\gamma\tilde{\gamma}}{s}\tr\M\boldsymbol{\Xi}^2\right]\EE\left[e^{\jmath u I_s(z)+\jmath u g(\Z)}\right]+\mathcal{O}\left(s^{-1}\right).
 	\label{eq:chi_3}
 	\end{equation}
 	The term $\chi_2$ can be dealt with in the same way, thus proving:
 	\begin{equation}
 	\chi_2=-z\EE\left[\stackrel{o}{\beta}e^{\jmath u I_s(z)+\jmath u g(\Z)}\right]\tilde{\gamma}\tr(\M\boldsymbol{\Xi}^2)+\mathcal{O}\left(s^{-1}\right).
 	\label{eq:chi2_first}
 	\end{equation}
 	Since $\tr(\M\boldsymbol{\Xi}^2)$ is of order $s$, we shall expand $\EE\left[\stackrel{o}{\beta}e^{\jmath u I_s(z)+\jmath u g(\Z)}\right]$ to at least the order $s^{-1}$, and thus $\stackrel{o}{\beta}$ and $\EE\left[e^{\jmath u I_s(z)+\jmath u g(\Z)}\right]$ cannot be separated in the same way as above.
 	
 	Indeed, we shall first take the sum over $j$ in \eqref{eq:useful_b}, thus yielding:
 	\begin{align}
 	\EE\left[\left[{\bf S}\Z\Z^{\herm}\right]_{p,p}e^{\jmath u I_s(z)+\jmath u g(\Z)}\right]&=\EE\left[s m_p\tilde{r}S_{p,p}e^{\jmath u I_s(z)+\jmath u g(\Z)}\right]-z\EE\left[\stackrel{o}{\beta}\tilde{r}\left[{\bf S}\Z\Z^{\herm}\right]_{p,p}e^{\jmath u I_s(z)+\jmath u g(\Z)}\right]\nonumber\\
 	&+\frac{z}{s}\EE\left[\jmath u \tilde{r}\left[{\bf S}\M{\bf S}\Z\Z^{\herm}\right]_{p,p}e^{\jmath u I_s(z)+\jmath u g(\Z)}\right]-\jmath u\EE\left[\tilde{r}\left[{\bf S}\M\right]_{p,p}e^{\jmath u I_s(z)+\jmath u g(\Z)}\right].\label{eq:previous}
 	\end{align}
 	Using the fact that:
 	$$
 	\frac{z}{s}\left[\left[{\bf S}\Z\Z^{\herm}\right]_{p,p}e^{\jmath u I_s(z)+\jmath u g(\Z)}\right]=\EE\left[e^{\jmath u I_s(z)+\jmath u g(\Z)}\right]-\EE\left[S_{p,p}e^{\jmath u I_s(z)+\jmath u g(\Z)}\right],
 	$$
 	Eq. \eqref{eq:previous} becomes:
 	\begin{align}
 	\EE\left[e^{\jmath u I_s(z)+\jmath u g(\Z)}\right]-\EE\left[S_{p,p}e^{\jmath u I_s(z)+\jmath u g(\Z)}\right]&=z\EE\left[m_p\tilde{r}S_{p,p}e^{\jmath u I_s(z)+\jmath u g(\Z)}\right]-z^2\EE\left[\stackrel{o}{\beta}\tilde{r}\left[\frac{{\bf S}\Z\Z^{\herm}}{s}\right]_{p,p}e^{\jmath u I_s(z)+\jmath u g(\Z)}\right]\nonumber\\
 	&+\frac{z^2}{s}\EE\left[\jmath u \tilde{r}\left[{\bf S}\M{\bf S}\frac{\Z\Z^{\herm}}{s}\right]_{p,p}e^{\jmath u I_s(z)+\jmath u g(\Z)}\right]-\frac{\jmath u z}{s}\EE\left[\tilde{r}\left[{\bf S}\M\right]_{p,p}e^{\jmath u I_s(z)+\jmath u g(\Z)}\right].\label{eq:solving_eq}
 	\end{align}
 	Solving $\EE\left[S_{p,p}e^{\jmath u I_s(z)+\jmath u g(\Z)}\right]$ in \eqref{eq:solving_eq} and using the relation $r_p=\frac{1}{1+zm_p\tilde{r}}$, we obtain:
 	\begin{align}
 	\EE\left[S_{p,p}e^{\jmath u I_s(z)+\jmath u g(\Z)}\right]&=\EE\left[r_pe^{\jmath u I_s(z)+\jmath u g(\Z)}\right]+\frac{z^2}{s}\EE\left[\stackrel{o}{\beta}r_p\tilde{r}\left[{\bf S}\Z\Z^{\herm}\right]_{p,p}e^{\jmath u I_s(z)+\jmath u g(\Z)}\right]\nonumber\\
 	&-\frac{z^2}{s}\EE\left[\jmath u \tilde{r}r_p\left[{\bf S}\M{\bf S}\frac{\Z\Z^{\herm}}{s}\right]_{p,p}e^{\jmath u I_s(z)+\jmath u g(\Z)}\right]+\frac{\jmath u z}{s}\EE\left[\tilde{r}r_p\left[{\bf S}\M\right]_{p,p}e^{\jmath u I_s(z)+\jmath u g(\Z)}\right].\label{eq:mult_eq}
 	\end{align}
 	Multiplying both sides in \eqref{eq:mult_eq} by $m_p$ and summing over $p$, we get:
 	\begin{align*}
 	\EE\left[\stackrel{o}{\beta}e^{\jmath u I_s(z)+\jmath u g(\Z)}\right]&=\EE\left[\frac{1}{s}\tr(\M\R-\M\EE{\bf S})e^{\jmath u I_s(z)+\jmath u g(\Z)}\right]+\frac{z^2}{s}\EE\left[\stackrel{o}{\beta}\frac{\tilde{r}}{s}\tr(\M\R{\bf S}\Z\Z^{\herm})e^{\jmath u I_s(z)+\jmath u g(\Z)}\right]\\
 	&-\frac{z^2}{s}\EE\left[\jmath u \tilde{r}\frac{1}{s}\tr(\M\R{\bf S}\M{\bf S}\frac{\Z\Z^{\herm}}{s})e^{\jmath u I_s(z)+\jmath u g(\Z)}\right]+\frac{\jmath u z}{s^2}\tilde{r}\EE\left[\tr\left(\M\R\M{\bf S}\right) e^{\jmath u I_s(z)+\jmath u g(\Z)}\right].
 	&\end{align*}
 	Using the approximating expressions in proposition \ref{prop:estimation}, we obtain:
 	\begin{align*}
 	\EE\left[\stackrel{o}{\beta}e^{\jmath u I_s(z)+\jmath u g(\Z)}\right]&=z^2\gamma\tilde{\gamma}\EE\left[\stackrel{o}{\beta}e^{\jmath u I_s(z)+\jmath u g(\Z)}\right]-\frac{z^2\tilde{\delta}\jmath u}{s(1-z^2\gamma\tilde{\gamma})}\left(\tilde{\delta}\frac{1}{s}\tr(\M^3\boldsymbol{\Xi}^3)-z\gamma^2\tilde{\gamma}\right)\EE\left[e^{\jmath u I_s(z)+\jmath u g(\Z)}\right]\\
 	&+\frac{\jmath u z}{s^2}\tilde{r}\EE\left[\tr(\M\R{\bf S}\M)e^{\jmath u I_s(z)+\jmath u g(\Z)}\right]+\mathcal{O}\left(s^{-2}\right).
 	\end{align*}
 	Hence,
 	\begin{align}
 	\EE\left[\stackrel{o}{\beta}e^{\jmath u I_s(z)+\jmath u g(\Z)}\right]&=-\frac{z^2\jmath u}{s(1-z^2\gamma\tilde{\gamma})^2}\left(\tilde{\gamma}\frac{1}{s}\tr(\M^3\boldsymbol{\Xi}^3)-z\gamma^2\tilde{\delta}^3\right)\EE\left[e^{\jmath u I_s(z)+\jmath u g(\Z)}\right]\nonumber\\
 	&+\frac{\jmath u z \tilde{\delta}\gamma}{s(1-z^2\gamma\tilde{\gamma})}\EE \left[e^{\jmath u I_s(z)+\jmath u g(\Z)}\right]+\mathcal{O}\left(s^{-2}\right).\label{eq:chi_2_useful}
 	\end{align}
Plugging \eqref{eq:chi_2_useful} into \eqref{eq:chi2_first}, the term $\chi_2$ can be written as:
 	\begin{align}
 	\chi_2&=\frac{z^3\jmath u \tilde{\gamma}}{s(1-z^2\gamma\tilde{\gamma})^2}\left(\tilde{\gamma}\frac{1}{s}\tr(\M^3\boldsymbol{\Xi}^3)-z\gamma^2\tilde{\delta}^3\right)\tr(\M\boldsymbol{\Xi}^2)\EE\left[e^{\jmath u I_s(z)+\jmath u g(\Z)}\right]\\
 	&-\frac{\jmath u z^2\gamma\tilde{\delta}^3}{(1-z^2\gamma\tilde{\gamma})}\frac{1}{s}\tr\left(\M\boldsymbol{\Xi}^2\right)\EE\left[e^{\jmath u I_s(z)+\jmath u g(\Z)}\right]+\mathcal{O}\left(s^{-1}\right).
 	\label{eq:chi_2}
 	\end{align}
 	Finally, it remains to deal with $\chi_4$. Using proposition \ref{prop:approximation}, we get:
 	\begin{equation}
 	\chi_4=-\frac{\jmath u \tilde{\delta}}{s}\tr\left(\M\boldsymbol{\Xi}^2\right)\EE\left[e^{\jmath u I_s(z)+\jmath u g(\Z)}\right]+\mathcal{O}\left(s^{-1}\right).
 	\label{eq:chi_4}
 	\end{equation}
 	Summing \eqref{eq:chi_1}, \eqref{eq:chi_3}, \eqref{eq:chi_2} and \eqref{eq:chi_4}, we obtain after some calculations:
 	\begin{align}
 	\EE\left[\tr\left(\frac{{\bf S}\Z\Z^{\herm}}{s}\right)e^{\jmath u I_s(z)+\jmath u g(\Z)}\right]&=\left[s\delta\tilde{\delta}+\frac{z^3\jmath u \tilde{\gamma}^2}{s(1-z^2\gamma\tilde{\gamma})^2}\frac{1}{s}\tr\left(\M^3\boldsymbol{\Xi}^3\right)\tr\left(\M\boldsymbol{\Xi}^2\right)+\frac{z\jmath u\tilde{\gamma}}{1-z^2\gamma\tilde{\gamma}}\frac{1}{s}\tr(\M^2{\boldsymbol{\Xi}}^3)\right.\nonumber\\
 	&\left.-\frac{z^2\jmath u \gamma\tilde{\delta}^3}{(1-z^2\gamma\tilde{\gamma})^2}\frac{1}{s}\tr(\M\boldsymbol{\Xi}^2)-\frac{\jmath u\tilde{\delta}}{1-z^2\gamma\tilde{\gamma}}\frac{1}{s}\tr\M\boldsymbol{\Xi}^2\right]\times\EE\left[e^{\jmath u I_s(z)+\jmath u g(\Z)}\right]+\mathcal{O}\left(s^{-1}\right).\label{eq:trace_previous}
 	\end{align}
 	Hence  the differential of $\Psi_s(u,z)$ with respect to $z$ satisfies:
 	\begin{align*}
 	\frac{\partial \Psi_s}{\partial z}&=\left[\frac{-u^2 z^3 \tilde{\gamma}^2}{(1-z^2\gamma\tilde{\gamma})^2}\frac{1}{s}\tr\left(\M^3\boldsymbol{\Xi}^3\right)\frac{1}{s}\tr\left(\M\boldsymbol{\Xi}^2\right)-\frac{u^2 z\tilde{\gamma}}{(1-z^2\gamma\tilde{\gamma})}\frac{1}{s}\tr\left(\M^2\boldsymbol{\Xi}^3\right)\right.\\
 	&\left.+\frac{ u^2 z^2\gamma\tilde{\delta}^3}{(1-z^2\gamma\tilde{\gamma})^2}\frac{1}{s}\tr\left(\M\boldsymbol{\Xi}^2\right)+\frac{u^2 \tilde{\delta}}{1-z^2\gamma\tilde{\gamma}}\frac{1}{s}\tr\left(\M\boldsymbol{\Xi}^2\right)\right]\Psi_s(u,z)+\mathcal{O}\left(s^{-1}\right).
 	\end{align*}
 	Following the same lines as in \cite{HAC06}, one can prove that:
 	\begin{equation}
 	-\frac{1}{2}\frac{d\log\left(1-z^2\gamma\tilde{\gamma}\right)}{dz}=\frac{1}{1-z^2\gamma\tilde{\gamma}}\left(-\frac{z^2\gamma\tilde{\delta}^3\frac{1}{s}\tr\left(\M\boldsymbol{\Xi}^2\right)}{1-z^2\gamma\tilde{\gamma}}+z\tilde{\gamma}\frac{1}{s}\tr\left(\M^2\boldsymbol{\Xi}^3\right)+\frac{z^3\tilde{\gamma}^2\frac{1}{s}\tr\left(\M^3\boldsymbol{\Xi}^3\right)\frac{1}{s}\tr\left(\M\boldsymbol{\Xi}^2\right)}{1-z^2\gamma\tilde{\gamma}}\right).
 	\label{eq:log1}
 	\end{equation}
 	Moreover, from the system of equations (51) in \cite{HAC06}, one can find that:
 	\begin{equation}
 	\frac{1}{2}\frac{d\log\tilde{\gamma}}{dz}=-\frac{\tilde{\delta}\frac{1}{s}\tr\left(\M\boldsymbol{\Xi}^2\right)}{1-z^2\gamma\tilde{\gamma}}.
 	\label{eq:log2}
 	\end{equation}
 	Using \eqref{eq:log1} and \eqref{eq:log2}, we finally get:
 	$$
 	\frac{\partial \Psi_s}{\partial z}=-\frac{u^2}{2}\left[-\frac{d}{dz}\log(1-z^2\gamma\tilde{\gamma})+\frac{d\log\tilde{\gamma}}{dz}\right]\Psi_s(u,z)+\mathcal{O}\left(s^{-1}\right).
 	$$
 	Let $\sigma_T^2=-\log\left(1-z^2\gamma\tilde{\gamma}\right)+\log\tilde{\gamma}$ and $K_s(u,z)=\Psi_s(u,z)\exp\left(\frac{u^2\sigma_T^2}{2}\right)$. Therefore, $K_s(u,z)$ satisfies:
 	$$
 	\frac{\partial K_s}{\partial z}=\epsilon(s,z)\exp\left(\frac{u^2\sigma_T^2}{2}\right),
 	$$
 	where it can be proven that $|\epsilon(s,z)|\leq \frac{K}{s}$, for every $s$ in $\left[0,\rho\right]$. On the other hand, we have:
 	$$
 	K_s(u,0)=\EE\left[e^{\jmath u \left(-\log\det\left(\frac{1}{s}\Z\Z^{\herm}-b_s\right)\right)}\right].
 	$$
 	Hence,
 	\begin{align*}
 	K_s(u,\rho)&=K_s(u,0)+\int_{0}^\rho \epsilon_s(u,x) dx \\
 	&=e^{\frac{u^2\log(1-\frac{p_t}{s})}{2}}+\mathcal{O}\left(s^{-1}\right).
 	\end{align*}
 	The characteristic function $\Psi_s(u,\rho)$ can be thus approximated as:
 	\begin{equation}
 	\Psi_s(u,\rho)=\exp\left({-\frac{u^2\sigma_T^2}{2}}+\frac{u^2\log(1-\frac{p_t}{s})}{2}\right)+\mathcal{O}\left(s^{-1}\right).
 	\label{eq:characteristic}
 	\end{equation}
 	The characteristic function satisfies the same equation as in \cite{HAC06}. The single difference is that the variance $\alpha_{N,t}(y)$ given by:
 	\begin{equation}
 	\alpha_{N,t}(y)=-\log\left(\frac{1-\gamma\tilde{\gamma}}{\tilde{\gamma}}\right)-\log(1-\frac{p_t}{s})
 	\label{eq:alpha}
 	\end{equation}
 	has two additive terms accounting for the variance of $g(\Z)$ and the correlation between $g(\Z)$ and $I_s(z)$. The CLT can be thus established by using the same arguments in \cite{HAC06}, provided that we show that $\lim\inf\alpha_{N,t}(y)> 0$. For that, we need only to prove that:
 	$$
 	\liminf_{s,p_t}\frac{1-z^2\gamma\tilde{\gamma}}{\tilde{\gamma}} >0.
 	$$
 	Deriving $\tilde{\delta}$ with respect to $z$, one can easily see that:
 	$$
 	\frac{1-z^2\gamma\tilde{\gamma}}{\tilde{\gamma}}=-\frac{1}{\frac{d\tilde{\delta}}{dz}}\frac{1}{s}\tr\left(\M\boldsymbol{\Xi}^2\right).
 	$$
 	It has been shown in \cite[eq.(67)]{HAC06} that $-\frac{d\tilde{\delta}}{dz}$ satisfies:
 	$$
 	0<-\frac{d\tilde{\delta}}{dz}<\frac{p_t}{s}\lambda_{{\rm max},t},
 	$$
 	where $\lambda_{\rm max}=\max\left(\lambda_{1,t},\cdots,\lambda_{p_t,t}\right)$. This fact combined with $\liminf\frac{1}{s}\tr\left(\M\boldsymbol{\Xi}^2\right)>0$ implies that $\liminf\alpha_{N,t}(y)> 0$.
 	It remains thus to express the variance $\alpha_{N,t}(y)$ using the original notations. One can easily show that:
 	\begin{align}
 	\delta&=\frac{1}{s}\tr\left(\frac{My}{s}\D_t^{-\frac{1}{2}}\U_t^{\herm}\H_t\H_t^{\herm}\U_t\D_t^{-\frac{1}{2}}+\frac{\I_N}{1+\delta}\right)^{-1}-\frac{(N-p_t)(1+\delta)}{s}\nonumber\\
 	&=\frac{1}{s}\tr\left((\G_t\G_t^{\herm})\left(\frac{My}{s}\H_t\H_t^{\herm}+\frac{\G_t\G_t^{\herm}}{1+\delta}\right)^{-1}\right)-\frac{(N-p_t)(1+\delta)}{s}\nonumber\\
 	&=\frac{1}{M}\tr\left((\G_t\G_t^{\herm})\left(y\H_t\H_t^{\herm}+\frac{s\, \G_t\G_t^{\herm}}{M(1+\delta)}\right)^{-1}\right)-\frac{(N-p_t)(1+\delta)}{s}.\label{eq:delta_function}
 	\end{align}
 	Then, from \eqref{eq:delta_function}, we can prove that $\frac{M(\delta+1)}{s}-1$ is solution in $x$ of:
 	\begin{equation}
 	x=\frac{1}{M}\tr\left((\G_t\G_t^{\herm})\left(y\H_t\H_t^{\herm}+\frac{\G_t\G_t}{1+x}\right)^{-1}\right).
 	\label{eq:kappa_x}
 	\end{equation}
 	Since $\kappa_t$ is the unique solution of \eqref{eq:kappa_x}, we have:
 	$$
 	\frac{M(\delta+1)}{s}-1=\kappa_t,
 	$$
 	or equivalently:
 	$$
 	\tilde{\delta}=\frac{1}{1+\delta}=\frac{M}{s(\kappa_t+1)}.
 	$$
 	Therefore:
 	\begin{equation}
 	\tilde{\gamma}=\tilde{\delta}^2=\frac{M^2}{s^2(\kappa_t+1)^2}.
 	\label{eq:gamma_tilde_o}
 	\end{equation}
 	In the same way, one can prove that $\gamma$ can be expressed in terms of the original notations as:
 	\begin{equation}
 	\gamma=\frac{s}{M^2}\tr\left(y\H_t\H_t^{\herm}\left(\G_t\G_t^{\herm}\right)^{-1}+\frac{\I_N}{\kappa_t+1}\right)^{-2}-\frac{(\kappa_t+1)^2s(N-p_t)}{M^2}.
 	\label{eq:gamma_o}
 	\end{equation}
 	Substituting \eqref{eq:gamma_o} and \eqref{eq:gamma_tilde_o} into \eqref{eq:alpha}, $\alpha_{N,t}(y)$ becomes
 	$$
 	\alpha_{N,t}(y)=\log M^2-\log\left((M-N)\left(M(\kappa_t+1)^2-\tr\left(y\H_t\H_t^{\herm}\left(\G_t\G_t^{\herm}\right)^{-1}+\frac{\I_N}{\kappa_t+1}\right)^{-2}\right)\right).
 	$$

 	\section{Proof of theorem \ref{th:y}}
 	\label{app:proofy}
 	1)
 	Denote by $R(y)$ and $f(y)$ the functionals given by:
 	\begin{align*}
 	f(y)&=\frac{1}{M}\tr(y\H_t\H_t^{\herm}\Q_t(y))+\frac{M-N}{M}-y\\
 	R(y)&=-\log\det(\Q_t(y))+(M-N)\log(y)-My.
 	\end{align*}
 	where $\Q_t(y)=\left(y\H_t\H_t^{\herm}+\frac{1}{M}\Y_t\Y_t^{\herm}\right)^{-1}$.
 	According to Poincar\'e-Nash inequality, we have:
 	\begin{equation}
 	{\rm var}(\yest)\leq K \sum_{i=1}^N\sum_{j=1}^M\left[ \EE\left|\frac{\partial \yest}{\partial Y_{i,j}^*}\right|^2+\EE\left|\frac{\partial \yest}{\partial Y_{i,j}}\right|^2\right]\ .
 	\label{eq:var}
 	\end{equation}
 	We only deal with the first sum in the previous inequality; the second one can be handled similarly.
 	By the implicit function theorem, if $\frac{\partial f}{\partial y}\neq 0$ then $\frac{\partial \yest}{\partial Y_{i,j}^*}$ writes:
 	\begin{equation}
 	\frac{\partial \yest}{\partial Y_{i,j}^*}=\frac{\frac{\partial f}{\partial Y_{i,j}^*}(\yest)}{\frac{\partial f}{\partial y}(\yest)}\ .
 	\label{eq:varb}
 	\end{equation}
 	As will be shown later, to conclude that ${\rm var}(\yest)=\mathcal{O}(M^{-2})$, we need to establish that $\left|\frac{\partial f}{\partial y}(\yest)\right|$ is lower bounded away from zero, which is a much stronger requirement than $\frac{\partial f}{\partial y}\neq 0$. This can be proved by noticing that $\frac{\partial R}{\partial y}=\frac{Mf}{y}$. Hence
 	\begin{equation}
 	\frac{\partial^2 R}{\partial y^2}(\yest)=\frac{M\frac{\partial f}{\partial y}(\yest)}{\yest}\ .
 	\label{eq:funeq}
 	\end{equation}
 	On the other hand, one can prove by straightforward calculations that $\left|\frac{\partial^2 R}{\partial y^2}(\yest)\right|\geq \frac{M-N}{\yest^2}$ which, plugged into \eqref{eq:funeq}, yields:
 	\begin{equation}
 	\left|\frac{\partial f}{\partial y}\right|\geq \frac{M-N}{M\yest}\ ,
 	\label{eq:fy}
 	\end{equation}
 	which is eventually uniformily lower bounded away from 0 due to Assumption \ref{ass:asymptotics} and to the fact that $\yest\le 1$ by mere definition.
 	Therefore,
 	\begin{align*}
 	\sum_{i=1}^N\sum_{j=1}^M \EE\left|\frac{\partial \yest}{\partial Y_{i,j}^*}\right|^2&\leq \frac{K}{M^4}\sum_{i=1}^N\sum_{j=1}^M |\left[\yest \Q_t\H_t \H_t^{\herm}\Q_t\Y\right]_{i,j}|^2 \ ,\\
 	&\leq \frac{K}{M^3}\tr\left(\Q_t\H_t\H_t^{\herm}\Q_t\frac{\Y\Y^*}{M}\Q_t\H_t \H_t^{\herm}\Q_t\right)\ ,\\
 	&\leq \frac{K}{M^2}\ .
 	\end{align*}
 	To prove 2), we rely on the resolvent identity which states:
 	\begin{equation}
 	\Q_t(a)-\Q_t(b)=(b-a)\Q_t(a)\H_t\H_t^{\herm}\Q_t(b)\ .
 	\label{eq:resolvent}
 	\end{equation}
 	Using \eqref{eq:resolvent}, we obtain:
 	\begin{align*}
 	\yest&=\frac{1}{M}(\yest-\EE \yest)\tr\H_t \H_t^{\herm}\Q_t(\yest)+\frac{1}{M}\tr\EE(\yest)\H_t\H_t^{\herm}\Q_t(\yest)+\frac{M-N}{M}\ ,\\
 	&=\frac{1}{M}(\yest-\EE \yest)\H_t\H_t^{\herm}\Q_t(\EE \yest)-\frac{1}{M}\tr(\yest-\EE \yest)^2\H_t\H_t^{\herm}\Q_t(\yest)\H_t\H_t^{\herm}\Q_t(\EE \yest)\\
 	&\quad +\frac{1}{M}\tr\EE(\yest)\H_t\H_t^{\herm}\Q_t(\EE \yest)-\frac{1}{M}\tr\EE(\yest)(\yest-\EE(\yest))\H_t\H_t^{\herm}\Q_t(\yest)\H_t\H_t^{\herm}\Q_t(\EE(\yest))+\frac{M-N}{M}\ ,\\
 	&\stackrel{(a)}{=}\frac{1}{M}(\yest-\EE \yest)\tr\H_t\H_t^{\herm}\T(\EE(\yest))+\frac{1}{M}\EE(\yest)\H_t\H_t^{\herm}\T(\EE(\yest))\\
 	&\quad -\EE(\yest)(\yest-\EE \yest)\EE\left[\frac{1}{M}\tr\H_t\H_t^{\herm}\Q_t(\yest)\H_t\H_t^{\herm}\Q_t(\EE(\yest))\right]+\frac{M-N}{M}+\varepsilon\ ,
 	\end{align*}
 	where $\varepsilon$ satisfies $\EE(\varepsilon)=\mathcal{O}(M^{-2})$. Note that equality $(a)$ follows from the fact that
 	$$
 	{\rm var}(\yest)=\mathcal{O}\left(\frac 1{M^{2}}\right)\quad \textrm{and}\quad {\rm var}\left(\frac{1}{M}\tr\H_t\H_t^{\herm}\Q_t(\yest)\H_t\H_t^{\herm}\Q_t(\EE(\yest))\right)=\mathcal{O}\left(\frac 1{M^{2}}\right)\ .
 	$$
 	Both estimates can be established with the help of Poincar\'e-Nash inequality.
 	Therefore:
 	\begin{align}
 	\EE(\yest)&=\frac{1}{M}\EE(\yest)\tr\H_t\H_t^{\herm}\T_t(\EE(\yest))+\frac{M-N}{M}+\mathcal{O}(M^{-2})\nonumber\\
 	&=1-\frac{1}{M(1+\kappa(\EE(\yest))}\tr((\G_t\G_t^{\herm})\T_t(\EE(\yest)))+\mathcal{O}(M^{-2})\nonumber\\
 	&=1-\frac{\kappa_t(\EE(\yest))}{1+\kappa(\EE(\yest))}+\mathcal{O}(M^{-2})\nonumber\\
 	&=\frac{1}{1+\kappa_t(\EE(\yest))}+\mathcal{O}(M^{-2})\ .\label{eq:yNt}
 	\end{align}
	 Using the mere definition of $y_{N,t}$ and \eqref{eq:yNt}, we obtain:
 	\begin{align}
 	\EE(\yest)-y_{N,t}&=\frac{1}{1+\kappa(\EE(\yest))}-\frac{1}{1+\kappa(y_{N,t})}+\mathcal{O}(M^{-2})\ ,\nonumber\\
 	&=\frac{\kappa(y_{N,t})-\kappa(\EE(y_{N,t}))}{(1+\kappa(\EE(\yest)))(1+\kappa(y_{N,t}))}+\mathcal{O}(M^{-2})\label{eq:final}\ .
 	\end{align}

	Following the same lines as in Appendix \ref{app:lemma_deter}, we can prove that for every real postives $y_1$ and $y_2$, we have:

	$$
	\kappa_t(y_1)-\kappa_t(y_2)=(y_2-y_1)\tilde{\alpha}_{N,t}(y_1,y_2)+(\kappa_t(y_1)-\kappa_t(y_2))\tilde{\beta}_{N,t}(y_1,y_2)
	$$
	where $\tilde{\alpha}_{N,t}(y_1,y_2)$ and $\tilde{\beta}_{N,t}(y_1,y_2)$ are given by:
	\begin{align*}
\tilde{\alpha}_{N,t}(y_1,y_2)&=\frac{1}{M}\tr\G_t\G_t^{\herm}\T(y_2)\H_t\H_t^{\herm}\T(y_1),\\
\tilde{\beta}_{N,t}(y_1,y_2)&=\frac{1}{M}\tr\frac{\G_t\G_t^{\herm}\T(y_2)\G_t\G_t^{\herm}\T(y_1)}{(1+\kappa_t(y_1))(1+\kappa_t(y_2))}.
	\end{align*}
	Moreover, we can easily notice that $\tilde{\beta}_{N,t} \leq \liminf \frac{N}{M} <1$. This allows us to express $\kappa_t(y_1)-\kappa_t(y_2)$ as:
	$$
\kappa_t(y_1)-\kappa_t(y_2)=\frac{\tilde{\alpha}_{N,t}(y_1,y_2)(y_2-y_1)}{1-\tilde{\beta}_{N,t}(y_1,y_2)}
	$$
Using this relation, we obtain from \eqref{eq:final}:
$$
\EE(\yest)-y_{N,t}=\gamma_{N,t}(\EE \yest-y_{N,t})+\mathcal{O}(M^{-2})
$$
where
$$
\gamma_{N,t}=\frac{\tilde{\alpha}_{N,t}(y_{N,t},\EE \yest)}{(1-\tilde{\beta}_{N,t}(y_{N,t},\yest))(1+\kappa_t(\EE\yest))(1+\kappa(y_{N,t}))}
$$
To conclude, we shall establish that $\limsup \gamma_{N,t}<1$. This is true, since 
using the relation $y_{N,t}=(1+\kappa(y_{N,t}))^{-1}$, we prove after some calculations that:
$$
\gamma_{N,t}=\frac{c_{N,t}}{1-d_{N,t}},
$$
where
\begin{align*}
c_{N,t}&=\frac{1}{M(1+\kappa(\EE(\yest))}\tr(\H_t\H_t^{\herm}+\G_t\G_t^{\herm})^{-1}\H_t\H_t^{\herm}\T(\EE(\yest))\G_t\G_t^{\herm}\\
d_{N,t}&=\frac{1}{M(1+\kappa(\EE(\yest)))}\tr(\H_t\H_t^{\herm}+\G_t\G_t^{\herm})^{-1}\G_t\G_t^{\herm}\T(\EE(\yest))\G_t\G_t^{\herm}
\end{align*}
Since $c_{N,t}+d_{N,t}\leq \frac{N}{M}$, and  $d_{N,t}\leq \frac{N}{M}$
$$
\frac{c_{N,t}}{1-d_{N,t}}-1=\frac{c_{N,t}+d_{N,t}-1}{1-d_{N,t}}\leq \frac{\frac{N}{M}-1}{1-\frac{N}{M}}<0,
$$
which implies that:
$$
\limsup_{N,M} \frac{c_{N,t}}{1-d_{N,t}}< 1.
$$
%
 	
 	\bibliographystyle{IEEEtran}
 	\bibliography{./IEEEabrv,./IEEEconf.bib,./tutorial_RMT.bib}

 	\end{document}